\documentclass{article}
\usepackage{graphicx,amsmath,amssymb,geometry,bbm,amsthm}
\usepackage{xcolor,comment}
\usepackage{enumitem}

\newtheorem{theorem}{Theorem}[section]

\newtheorem{lemma}[theorem]{Lemma}
\newtheorem{definition}[theorem]{Definition}
\newtheorem{proposition}[theorem]{Proposition}
\title{}
\newcommand{\sa}[1]{\ensuremath{\,{\buildrel #1 \over \longleftrightarrow}\,}}
\newcommand{\Z}{\mathbb{Z}}
\newcommand{\prob}{\mathbb{P}}

\newcommand{\cC}{\mathcal{C}}
\newcommand{\cD}{\mathcal{D}}
\newcommand{\lra}{\leftrightarrow}
\newcommand{\fC}{\mathfrak{C}}
\newcommand{\E}{\mathbb{E}}
\newcommand{\edges}{\mathsf{E}}
\newcommand{\cutp}{\beta}

\newcommand{\tE}{\widetilde{E}}

\makeatletter
\newcommand\xleftrightarrow[2][]{%
  \ext@arrow 9999{\longleftrightarrowfill@}{#1}{#2}}
\newcommand\longleftrightarrowfill@{%
  \arrowfill@\leftarrow\relbar\rightarrow}
\makeatother
\begin{document}
    \title{Robust construction of the incipient infinite cluster in high dimensional critical percolation}
   \author{Shirshendu Chatterjee \thanks{Email: shirshendu@ccny.cuny.edu.} \\ \small{The City College of New York and CUNY Graduate Center}  \and Pranav Chinmay \thanks{Email: pchinmay@gradcenter.cuny.edu.} \\ \small{CUNY Graduate Center}  \and Jack Hanson \thanks{Email: jhanson@ccny.cuny.edu.}\\ \small{The City College of New York and CUNY Graduate Center} \and Philippe Sosoe \thanks{Email: ps934@cornell.edu.} \\ \small{Cornell University} }
	\maketitle

 \begin{abstract}
        We give a new construction of the incipient infinite cluster (IIC) associated with high-dimensional percolation  in a broad setting and under minimal assumptions.
Our arguments differ substantially from  earlier constructions of the IIC; we do not directly use   the machinery of the lace expansion or similar diagrammatic expansions.  
We  show that the IIC  may be constructed by conditioning on the cluster of a vertex being infinite in the supercritical regime $p > p_c$ and then taking $p \searrow p_c$. Furthermore, at criticality, we show that the IIC may be constructed by conditioning on a connection to an arbitrary distant set $V$, generalizing previous constructions where one conditions on a connection to a single distant vertex or the boundary of a large box.

The input to our proof are the asymptotics for the two-point function obtained by Hara, van der Hofstad, and Slade. Our construction thus applies in all dimensions for which those asymptotics are known, rather than an unspecified high dimension considered in previous works. The results in this paper will be instrumental in upcoming work related to structural properties and  scaling limits of various objects involving high-dimensional percolation clusters at and near criticality.
    \end{abstract}

	\section{Introduction}

 Bernoulli percolation is a fundamental model of mathematical statistical mechanics. Despite its apparent simplicity, it exhibits many of the qualitative features of more complex statistical mechanical models that undergo a continuous phase transition. Of particular interest is the behavior of the model in a neighborhood of the critical point, the value of the model parameter above which an infinite cluster appears on the lattice. It is widely expected that, at the critical point in Bernoulli percolation in $\mathbb{Z}^d$, there is no infinite cluster, a fact referred to as continuity of the phase transition.  Continuity has been proved in dimensions $d=2$ by T. E. Harris \cite{Har}  and H.~Kesten \cite{K0} and $d \ge 19$ by T. Hara and G. Slade \cite{HS};  the latter result was extended to the case $d\ge 11$ by R. Fitzner and R. van der Hofstad \cite{FH,FH2}. Obtaining a proof for continuity in intermediate dimensions remains a central open problem in percolation theory.

 We begin with a discussion of the main object of this work, the \emph{incipient infinite cluster}, as introduced in \cite{K}.  For this purpose, consider the integer lattice $\mathbb{Z}^d$ as a graph with edges between points differing by $\pm 1$ along one of the coordinates; we denote the corresponding edge set by $\edges(\Z^d)$. We let $\prob_p$ denote the infinite product of Bernoulli measures with parameter $p$ on $\{0,1\}^{\edges(\Z^d)}$. The critical density $p_c$ is defined by
\[p_c= \inf\{p: \prob_p(\text{the component of 0 is infinite})>0\},\]
so that $0<p_c<1$. By ``component", we mean the lattice points in $\mathbb{Z}^d$ connected to the origin $0:=(0,\ldots,0)$ in $\mathbb{Z}^d$ by a path of open edges (edges whose corresponding Bernoulli variable equals one). 

 At the critical point $p=p_c$, the mean cluster volume diverges, implying that any realization of the process contains large clusters. It is therefore natural to study typical realizations where the origin belongs to a large open cluster. Intuitively, this corresponds to conditioning on the origin being part of a cluster of increasing size and then taking the limit as the cluster size approaches infinity. Insights on the fractal properties of large critical clusters originating in the physics literature, such as the influential work of S. Alexander and R. Orbach on random walks and diffusion in critical percolation \cite{AO}, are most naturally formulated in mathematical terms by considering such a limit.
 
Such considerations inspired H. Kesten \cite{K} to give a precise meaning to the limit in the case of two dimensional percolation. Let $B(n)=\mathbb{Z}^2\cap[-n,n]^2$ be a box of side-length $2n+1$ centered at the origin, and let $\{0\leftrightarrow \partial B(n)\}$ be the event that the origin $(0,0)\in \mathbb{Z}^2$, (abbreviated here as $0$) is connected to the boundary $\partial B(n)=\{x\in \mathbb{Z}^2:\max\{|x_1|,|x_2|\}=n\}$ of $B(n)$ by a path of open edges. Kesten showed that the limit
 \begin{equation}
 \label{eqn: kesten-def}
 \nu(E)=\lim_{n\rightarrow\infty }\mathbb{P}(E\mid 0\leftrightarrow \partial B(n))
 \end{equation}
 exists for events $E$ depending only on the status of edges in a fixed box centered at the origin. Moreover, the measure $\nu$ coincides with another limit, where one varies the parameter $p$ instead:
 \begin{equation}\label{eqn: p-def}
 \nu(E)=\lim_{p\searrow p_c} \mathbb{P}_{p}(E \mid  0 \lra \infty).
 \end{equation}
Here, $\{0\leftrightarrow \infty\}$ is the event that the cluster of the origin $0$ is infinite. 
Under the measure $\nu$, the origin is contained in an infinite cluster with probability one. This cluster or, interchangeably, the measure $\nu$ itself is called an \emph{incipient infinite cluster} (IIC). The terminology is meant to suggest that large clusters with a structure locally similar to the IIC ``join together'' to form an infinite cluster once one passes the threshold $p_c$. 

\subsection{Main results}
Kesten's proof and subsequent results on the IIC in dimension $d=2$ such as \cite{J} and \cite{DS} rely on Russo-Seymour-Welsh (RSW) crossing estimates, as well as a certain renewal or domain Markov property for circuits in two dimensions. These tools are not available in general dimension, and so his proof does not at first appear to apply beyond the two-dimensional setting. The main contribution of this paper is to show how the cluster extension techniques developed in \cite{KN} and refined in \cite{CH,CHS} can replace the gluing techniques based on RSW and FKG inequalities in implementing the general scheme of Kesten's argument in the high-dimensional setting. 

We give a new, general construction of the incipient infinite cluster in critical, high-dimensional percolation. For the statement of our results and for the remainder of the paper, we broaden our choice of edge sets: $\edges(\Z^d)$ may be either the nearest-neighbor edge set \eqref{eq:nnset} as above, or the spread-out edge set defined below at \eqref{eq:spreadout}.  Our main result is the following:

		\begin{theorem}
			\label{thm:main}
            Let $\prob$ denote the product Bernoulli measure on $\{0,1\}^{\edges(\mathbb{Z}^d)}$, with common parameter $p=p_c$, the critical density. Let also $d>6$ be such that the estimates \eqref{2pt} hold. (For example, $d\ge11$ suffices \cite{FH, FH2}.)

			Suppose that, for each $n$, the sets $V_n, D_n$ are subsets of $\Z^d$ satisfying $(V_n \cup D_n) \cap B(n) = \varnothing$ and such that $0$ and some vertex of $V_n$ are in the same connected component of $\Z^d \setminus D_n$. Then we have
			\[\lim_{n \to \infty} \prob( \cdot \mid 0 \sa{\Z^d \setminus D_n} V_n )  =: \nu(\cdot) \]
			exists, where the limit is in the sense of weak convergence.
			\end{theorem}

        The estimates \eqref{2pt} are known to hold for the nearest-neighbor model, where they are due to Hara \cite{H} when  $d\ge 11$, and for $d>6$ in the spread-out case, where they were proved by Hara-van der Hofstad-Slade \cite{HHS}. They are expected to hold for $d>6$, and, consequently, so is Theorem \ref{thm:main}.

		We note that the measure $\nu$ appearing in Theorem~\ref{thm:main} must be independent of the choice of $V_n$ and $D_n$. To see this, given two pairs of sequences $(V_n), (D_n)$ and $(V_n'), (D_n')$, interleave them by setting $V_n'' = V_n$ for $n$ odd and $V_n'' = V_n'$ for $n$ even, with a similar definition for $D_n''$. Applying the theorem to the sequences $(V_n''), (D_n'')$ shows the claimed equality of limits.
		In particular, taking $V_n = n e_1$ and $D_n = \varnothing$, we see that $\nu$ is the IIC measure constructed by van der Hofstad and J\'arai in \cite{HJ}. 

        As explained above, in his original paper \cite{K}, Kesten gave an alternate construction of the incipient infinite cluster starting from supercriticality. Our other main result extends this construction to the high-dimensional setting: 
        \begin{theorem}
            \label{thm:main2}
            Let $d>6$ be such that the estimates \eqref{2pt} hold, and let $\{0\leftrightarrow\infty\}$ denote the event that the cluster of the origin in $\mathbb{Z}^d$ is unbounded. 
            Let $\nu$ be the IIC measure from the statement of Theorem~\ref{thm:main}. We have
            \[\lim_{p \searrow p_c} \prob_p(\cdot \mid 0 \lra \infty) = \nu(\cdot),\]
            where the limit is in the weak sense.
        \end{theorem}
        Both theorems are proved using the same basic building blocks. We prove many of our lemmas in the general setting of $p \geq p_c$, with an eye toward using them in the proofs of both Theorem~\ref{thm:main} and Theorem~\ref{thm:main2}.

\subsection{Previous results and motivation}

Kesten's work \cite{K}, as well as his study of the random walk \cite{K1} on the incipient infinite cluster, identified the IIC as an object of central importance in the study of critical percolation. Here we give a necessarily partial and limited survey of the subsequent literature. 

M. Damron and A. Sapozhnikov \cite{DS} have adapted Kesten's argument to constructing versions of the two-dimensional IIC where one conditions on the existence of multiple arms with prescribed status (open or closed) extending from the origin to infinity. See also the work by D. Basu and A. Sapozhnikov \cite{BS} for a general version of Kesten's argument that applies to percolation on graphs where the infinite cluster is unique and whose crossing probabilities satisfy a suitable quasimultiplicativity property.

Alternate constructions of an infinite cluster mimicking the asymptotic behavior of large critical percolation clusters in two dimensions were proposed by M. Aizenman \cite{A} and C. Borgs, J. Chayes, H. Kesten and J. Spencer \cite{BCKS}. We briefly review them here.

A spanning cluster of the box $B_d(n)=\mathbb{Z}^d\cap [-n,n]^d$ in $d$-dimensional percolation is an open cluster that intersects both the left side $\{-n\}\times [-n,n]^{d-1}$ and the right side $\{n\}\times[-n,n]^{d-1}$
of $B_d(n)$. Aizenman proposed (in general dimension) to construct the IIC at $p_c$ by selecting a site $I_n$ uniformly at random from all spanning clusters in $B_d(n)$ and recentering around it before taking the large $n$ limit. In \cite{BCKS}, the authors introduce a different procedure, choosing instead a vertex uniformly from the $k$th largest cluster of $B_d(n)$, with $k$ fixed, and recentering around this vertex. A.~A.~J\'arai \cite{A} showed in dimension $d=2$ that both of these procedures lead to the same limit as Kesten's IIC \eqref{eqn: kesten-def}.
 
Yet another construction of an IIC in two dimensions was given by J. Chayes, L. Chayes and R. Durrett in \cite{CCD}. These authors introduced an inhomogenous independent percolation model where the percolation density decays to $p_c$ as the distance from the origin increases: an edge $e\in \edges(\Z^2)$ is open with probability
\[p(e)=p_c+q(|e|),\]
where $q(t)\rightarrow 0$ as $t\rightarrow \infty$. It is shown in \cite{CCD} that if $q(t)$ decays sufficiently slowly, there is an infinite cluster at the origin with positive probability. This inhomogeneous model, with $q(e)=(1+|e|)^{-\lambda}$ and $\lambda=1/\nu$, where $\nu$ is the correlation length exponent and one conditions that $0$ be in an infinite cluster was proposed in \cite{CCD} as an interpretation of the IIC. For inhomogenous models with general $q$ decreasing to zero, A.~A.~J\'arai showed that choosing a vertex in the cluster of the origin uniformly in a box of size $n$ of the inhomogeneous percolation and re-centering the cluster around this vertex leads to the IIC in the $n\rightarrow \infty$ limit.

In high dimensions, T. Hara and G. Slade \cite{HS,HS2} studied the incipient infinite cluster by interpreting the informal description given above differently. They consider the cluster of the origin conditioned to have cardinality $n$. The estimates for the two-point function in \cite{HS-2pt}, previously obtained by the same authors through the lace expansion, imply that the typical radius of the cluster under this conditioning is $n^{1/4}$. Hara and Slade derive the limits of the conditioned two- and three- point functions with spatial rescaling by $n^{-1/4}$ and show that these coincide with those of a continuum process known as the integrated super-Brownian excursion. This provides a strong indication of the existence of a scaling limit for the IIC.

 Later, van der Hofstad and J\'arai \cite{HJ} studied a limit more closely resembling \eqref{eqn: kesten-def} in the high-dimensional case. They used lace expansion techniques to show that the limit
 \begin{equation}\label{eqn: HJ}
 \lim_{|x|\rightarrow \infty} \mathbb{P}(E\mid 0\leftrightarrow x)
 \end{equation}
 exists for spread-out percolation models in dimensions $d>6$, or in sufficiently high dimensions for nearest-neighbor percolation. The dimensions to which their result applies in the nearest-neighbor case are not specified. Van der Hofstad and J\'arai's approach was revisited and significantly extended by M. Heydenreich, van der Hofstad and T. Hulshof \cite{HHH}, who use a refined lace expansion analysis to show that the limit \eqref{eqn: HJ} exists under the so-called triangle condition. Their arguments apply both to finite-range percolation models in sufficiently high dimensions and some long-range models, where their results are conditional on the existence of the one-arm exponent, i.e. the analog of \eqref{KN} in the corresponding model. They also prove that there is a subsequence $n_k$ such that the limit \eqref{eqn: kesten-def} exists and coincides with \eqref{eqn: HJ} for every cylinder event $E$. Under the unproven assumption that the one-arm probability \eqref{KN}, scaled by the appropriate power of $n$ (power 2 for finite-range Bernoulli percolation, as in our case), has a limit, they show that \eqref{eqn: HJ} exists along the integers. As a consequence of our main result, we show unconditionally that \eqref{eqn: kesten-def} and \eqref{eqn: HJ} exist whenever \eqref{2pt} holds above $d=6$.

 In this context, we mention recent innovations  of H. Duminil-Copin and R. Panis \cite{DCP}, who derive upper bounds, uniform in the subcritical region up to the critical point, for the two-point function in spread-out percolation for dimensions $d$ greater than 6. These bounds imply in particular the finiteness of Aizenman and C. M. Newman's famous triangle diagram, a hallmark of mean field behavior for this model.

        The main result, Theorem \ref{thm:main} differs from previous results not just in the strength of the conclusion, but also in its proof. We do not directly rely on lace expansion techniques as in \cite{HJ}. Instead, we take the estimates \eqref{2pt} as input and implement Kesten's Markov-Chain relaxation argument from \cite{K} using only gluing techniques as in \cite{KN,CH,CHS}.
        
        \paragraph{Future and ongoing work.} As already mentioned, the present work is motivated by a desire to complete the picture of Heydenreich-van der Hofstad-Hulshof by placing the high-dimensional and two-dimensional IIC constructions on a more equal footing.   A further reason for our interest in revisiting the construction of the incipient infinite cluster in high dimensions in the general setting of Theorem
        \ref{thm:main} stems from our forthcoming work on the scaling of the chemical distance. We have previously proved \cite{CHS} that in critical percolation in high dimensions (defined by the validity of the mean field bounds \eqref{2pt}) the distance from the origin to the point $(n,0,\ldots,0)$, conditioned on the existence of a connection, is tight on the scale $n^2$, with precise control of the tail behavior. A natural next step is then to study the asymptotic distribution of the rescaled distance, a task we undertake in an upcoming paper. The generality of Theorem \ref{thm:main} plays a key role in the resulting moment computation.

        \subsection{Outline of the proof}
        Kesten's idea in \cite{K} is to approximate the probability 
        \begin{equation}\label{eqn:PE}
        \mathbb{P}(E,0\leftrightarrow \partial B(n))
        \end{equation}
        by a product of the form
        \begin{equation} \label{eqn: decomp}
        \sum_{C_1,\ldots,C_{K+1}} M_1(C_1,C_2)\cdots M_K(C_K,C_{K+1})\gamma(C_{K+1},n),
        \end{equation}
        where $C_i$ represent possible realizations of open circuits in a specific sequence of successive annuli $A_{i}$ lying between $0$ and $\partial B(n)$, and  each factor $M_i(C_i,C_{i+1})$ represents the probability that $C_{i+1}$ is the innermost open circuit in $A_{i+1}$ and there is an open connection from the $C_i$ to $C_{i+1}$ in the region between those two circuits. Kesten then derives the existence of the limit \eqref{eqn: kesten-def} from a ratio-limit theorem for (inhomogeneous) Markov Chains due to Hopf \cite{H}. 
        The argument just outlined, adapted to our setting, is recapitulated in Section  \ref{sec: MC-kesten}. Here we do not deviate much from \cite{K}.

        Significant technical difficulties arise in the justification of the approximation \eqref{eqn: decomp} in higher dimensions. In the two-dimensional setting \cite{K}, this decomposition is a relatively straightforward consequence of the existence of open circuits in annuli on sufficiently many scales. This is guaranteed by Russo-Seymour-Welsh estimates. The existence of open circuits and planarity in turn imply the existence of interfaces along which one can decouple the probabilities of connection events inside and outside, leading to the form \eqref{eqn: decomp}.

        In the high-dimensional setting, we replace circuits with certain (unique) spanning clusters of annuli. The clusters in question have to be carefully defined to ensure their uniqueness, high probability of occurrence, and good connectivity properties.  We defined the clusters in such a way that they must be crossed by any connection from the origin to a large distance. We also demand of the spanning clusters that they have certain good properties (see Definition \ref{defin:Kgood}), including that the volumes of the inner and outer boundary vertices of the cluster be typical and that most of these vertices be  regular in the sense of \cite{KN}. The last condition ensures the possibility of extending clusters.
     We use it in Section \ref{sec: structural} to derive connection probability estimates between good clusters, as well as between clusters and distant points. The key result here is the gluing Lemma \ref{lem:glue}. This lemma plays a key role in verifying the hypothesis \eqref{eq:kappasquare} in Lemma \ref{lem:hopflem}, Hopf's Markov chain convergence estimate.
     
        The rest of Section \ref{sec: structural} is devoted to proving that with high probability, a connection from the origin to a large distance which passes through a good cluster on a given scale $k_i$ extends to scale $k_{i+1}$ by passing through another good cluster, without "backtracking". (See Definition \ref{defin:verygood2} and Lemma \ref{lem:goodlikely} for precise definitions.) This decomposition of the connection from the origin to a large distance into a string of connections between good spanning clusters is the main input to Lemma \ref{lem:matprod}, the approximation of the event \eqref{eqn:PE} by a matrix product to which Hopf's Markov convergence result may be applied.
        
		\section{Preliminaries and definitions}
    We set the stage for the remainder of the paper. In Section~\ref{sec:defs}, we introduce some basic notation to which the remainder of the paper will refer, as well as some existing results which we will use in our estimates. We then in Section~\ref{sec:fixing} fix a particular event whose probability will be shown to converge. This will allow us to define a particular base length scale which will be the basis for our constructions throughout the paper. Finally, in Section~\ref{sec:regu}, we introduce a notion of regularity originating from \cite{KN}, and we describe how it will be useful in the following arguments.

        \subsection{Notation and useful results \label{sec:defs}}
        We consider the lattice with vertex set $\Z^d$ and either the nearest-neighbor edge set
        \begin{equation} \label{eq:nnset} \{ \{x, y\}: \, \|x - y\|_1 = 1\} \end{equation}
        or the \emph{spread-out} edge set with parameter $\Lambda \geq 1$:
        \begin{equation}\label{eq:spreadout}
            \{ \{x, y\}: \, 0 < \|x - y\|_\infty \leq \Lambda  \}\ . 
        \end{equation}
        In either case, we write $\edges(\Z^d)$ for the edge set under consideration. We sometimes identify the graph $(\Z^d, \edges(\Z^d))$ with its vertex set $\Z^d$ when this will not cause confusion, and we occasionally write $\edges(A)$ for the set of edges of $\Z^d$ with both endpoints in a vertex set $A$. We write $\Omega = \{0, 1\}^{\edges(\Z^d)}$ for the canonical probability space for the model, endowed with its Borel sigma-algebra. A typical element of $\Omega$ will be denoted by $\omega$. As a convention, we take $e \in \edges(\Z^d)$ to be open if and only if $\omega_e = 1$. We write $\prob_p$ for the probability measure under which the variables $(\omega_e)$ are i.i.d.~Bernoulli$(p)$ random variables. The \emph{open subgraph} is the random subgraph of $(\Z^d, \edges(\Z^d))$ induced by the edges which are open in $\omega$. If $A$ is a subgraph of $\Z^d$, the open subgraph of $A$ is defined analogously.

         The symbol $|\cdot|$ denotes the Euclidean norm, and other $p$-norms on $\Z^d$ are written using the symbol $\|\cdot\|_p$. We often sum powers of Euclidean norms, and to avoid notational complications in such sums, we introduce the convention that $|0|^{a} = 1$ for any negative power $a$.

		It is useful to introduce notation for sub-boxes of $\Z^d$, writing $B(x; r)$ for the set $x + [-r, r]^d$. An \emph{annulus} $A$ is a set of the form
        \begin{equation}\label{eq:anndef}
                    A = B(x; s) \setminus B(x; r) \ .
        \end{equation}
        for $s > r$. We also define a singleton $\{x\}$ for $x \in \Z^d$ to be a degenerate special instance of an annulus  with $s = 0$ and $r = -1$. We introduce notation for the inner and outer boundaries of an annulus: if $A$ is the annulus in \eqref{eq:anndef}, we set
        \begin{equation}
         \label{eq:annbdy}
        \begin{split}
            \partial_{in} A &:= \{y \in A: \, \{y, z\} \in \mathcal{E} \text{ for some $z \in B(x;r)$} \}\ , \quad \text{and} \\
		  \partial_{out} A &:= \{y \in A: \, \{y, z\} \in \mathcal{E} \text{ for some $z \notin B(x;s)$} \}\ .
          \end{split}
          \end{equation}
        
        We write $\fC(x)$ for the open cluster of a vertex $x \in \Z^d$ --- in other words, the component of the open subgraph containing $x$. We will sometimes consider this as a subgraph with the associated edges, and we will sometimes consider it as a vertex set; the precise meaning will be clear from context.   
        The symbol $p_c$, as usual, denotes the critical probability for the model: the element of $(0, 1)$ such that, $\prob_p(\exists \text{ an infinite open cluster})$ is $1$ for $p > p_c$ and $0$ for $p < p_c$.

        A useful notion will be that of \emph{restricted} connections. Given a subgraph $A$ (or vertex set $A$, which we identify here with the subgraph of $(\Z^d, \edges(\Z^d))$ it induces), we let $\fC_A(x)$ denote the component of the open subgraph of $A$ containing $x$. We write $x \sa{A} y$ to mean that $y \in \fC_A(x)$.
		
		Some crucial results we will use are the following. First, we have the Hara-van der Hofstad-Slade bounds for the two-point function. This powerful result, proved in \cite{H,HHS}, says that there exist constants $C > c > 0$ so that for all $x \in \mathbb{Z}^d$,
		\begin{equation}
		c|x|^{2-d} \leq \mathbb{P}_{p_c}(0 \leftrightarrow x) \leq C |x|^{2-d},
		\label{2pt}
		\end{equation}
        in nearest-neighbor Bernoulli percolation when $d\ge 11$, and in sufficiently spread-out models when $d>6$. 
        Other results we quote about high-dimensional percolation (e.g.~\eqref{KN} below) all hold for any $d > 6$ and either choice of $\edges(\Z^d)$ under the assumption that \eqref{2pt} is valid. Thus \eqref{2pt} is the main input our arguments will require.
    
        The result \eqref{2pt} was extended in \cite{CH} to restricted connections for vertices far from the boundary: for all $\epsilon \in (0, 1)$, there exists a constant $c = c(\epsilon) > 0$ such that the following holds:
        \begin{equation}
            \label{eq:restrtwopt}
            \begin{gathered}
                \text{for all $n$, and for all $x \in B((1-\epsilon) n)$, we have}\\
                c|x|^{2-d} \leq \mathbb{P}_{p_c}(0 \sa{B(n)} x)\ .
            \end{gathered}
        \end{equation}
        (The analogous upper bound to the lower-bound \eqref{eq:restrtwopt} follows trivially from \eqref{2pt}.)
		
		We will also use Kozma-Nachmias' result on the arm exponent, obtained in \cite{KN}. There exist constants $C > c > 0$ so that for all $n \geq 1$, 
		\begin{equation}
		cn^{-2} \leq \mathbb{P}_{p_c}(0 \leftrightarrow \partial B(n)) \leq Cn^{-2}.
		\label{KN}
		\end{equation}
        A final basic power law estimate we use comes from \cite{CH}. It controls the probability of a restricted connection to a vertex on the boundary of a box: there exists a $C > 0$ such that 
        \begin{equation}
            \label{eq:hstwopt}
            \text{for all $n$ and all $x$ with $\|x\| = n$,} \quad \prob(0 \sa{B(n)} x) \leq C n^{1-d}\ .
        \end{equation}

	
		
	    Given an annulus $A$ and a subgraph $F$ of $A$ (considered as an induced subgraph of $(\Z^d, \edges(\Z^d))$, we say $F$ is a \emph{spanning set} of $A$ if $F$ contains vertices of both $\partial_{in} A$ and $\partial_{out} A$. We extend this notion to vertex sets in the natural way.  A \emph{spanning cluster} is a component of the open subgraph of $A$ that is also a spanning set. We note that if $F$ is a subgraph of $A$, the event $\{F \text{ is a spanning cluster}\}$ is measurable with respect to the edges of $A$ which have an endpoint in the vertex set of $F$ --- in particular, we need not examine edges having an endpoint outside of $A$.

        For spanning sets $F$ of an annulus $A$, we introduce the notation $\partial_{in}^A F$ (resp.~$\partial_{out}^A F$) for the set of vertices in $F$ which lie in $\partial_{in} A$ (resp.~$\partial_{out} A$). This notion depends on the choice of ambient annulus $A$, but we write $\partial_{in}$ instead of $\partial_{in}^A$ when the annulus is clear in context, and take a similar convention for $\partial_{out}$.

        We will also have use of the following lemma controlling the probability an open cluster of a small region connects to the boundary of a larger region. This may be viewed as a slight rephrasing of \cite[Lemma 3.2]{CH}, and appears nearly verbatim in the second-to-last inequality of the proof of that lemma. 
        \begin{lemma}
            \label{lem:nofurther}
            Let $A_0 \subseteq A_1$ be two finite subgraphs of $\Z^d$, and suppose $\cC$ is a connected subgraph of $A_0$.  We let $\partial_{A_1} \cC$ be the set of vertices $v \in A_1 \setminus A_0$ such that $\{v, w\} \in \edges(\Z^d)$ for some vertex $w \in \cC$.  Then for each vertex set $B \subseteq A_1$, we have
            \[\prob_p\left(\cC \sa{A_1} B \mid \cC \text{ is an open cluster in the graph $A_0$}\right) \leq \sum_{w \in \partial_{A_1} \cC} \prob_p(w \sa{A_1 \setminus \cC} B)\ . \]
        \end{lemma}

        Finally, we will have use for the following  estimate for convolutions of power functions. We omit the elementary proof.
        \begin{lemma} Let $0<a,b$ be such that $0<a+b<d$ and $x\neq y\in \mathbb{Z}^d$. There is a constant $C>0$ such that
        \begin{equation}\label{eqn: convolution}\sum_{z\in \mathbb{Z}^d} \frac{1}{|z-x|^{d-a} |z-y|^{d-b}}\le \frac{C}{|x-y|^{d-a-b}}.
        \end{equation}
        \end{lemma}

        {\bf A note about constants.} In addition to special recurring constants which are kept track of for a stretch of time throughout the paper (like $K_0$ from Lemma~\ref{lem:reg}), there are various constants that appear in our estimates whose exact value will not be important. We generally use the symbols $C$ and $c$ to represent a generic positive constant whose value may change from line to line. Unless otherwise specified, the value of such a constant depends only on the lattice under consideration and not on any other parameters introduced in the course of the arguments.

        One special  case worth noting is the fixed choice of $K_2$ that will be made as the output of Lemma~\ref{lem:glue} and specifically remarked upon at \eqref{eq:Kfixed}. From that point onward, the value $K_2$ will be carried forward in any arguments relying on the output of Lemma~\ref{lem:glue}, and hence this constant will be recalled multiple times in the latter half of the paper. 
        
        \subsection{Initial preparation \label{sec:fixing}}
        To prove Theorems~\ref{thm:main} and \ref{thm:main2}, it will suffice to show the convergence of the probability of a single cylinder event. In what follows, for the proof of Theorem~\ref{thm:main}, we fix a sequence $(V_n, D_n)$ as described in the statement of that theorem. For the proofs of both theorems we fix a particular cylinder event:
        \begin{equation}
		\label{eq:ee}
		\text{with } L \geq 1  \text{ fixed, we let }  E \text{ be an arbitrary event measurable with respect to } (\omega_e)_{e \in B(2^L)}\ .
		\end{equation}
        For this choice of $E$ and $L$, we will show the following results:
        \begin{equation}
            \label{eq:thm1}
            \nu(E) := \lim_{n \to \infty} \prob( E \mid 0 \sa{\Z^d \setminus D_n} V_n ) \quad \text{exists};
        \end{equation}
        \begin{equation}
            \label{eq:thm2}
            \lim_{p \searrow p_c}  \prob_p( E \mid 0 \lra \infty ) = \nu(E), \text{ the quantity from \eqref{eq:thm1}}.
        \end{equation}
        Theorem~\ref{thm:main} follows from \eqref{eq:thm1} via a standard extension argument similar to the one in \cite{K}. Then Theorem~\ref{thm:main2} follows from \eqref{eq:thm2} similarly. {\bf We therefore devote our efforts to showing \eqref{eq:thm1} and \eqref{eq:thm2}, with $E$, $(V_n)$, and $(D_n)$ fixed.}

        Given an initial value of $L$ as in \eqref{eq:ee}, we  
        \begin{equation}
            \label{eq:k0choice}
            \text{choose a large parameter } k_1 \geq  L + \Lambda +  64d^4 + 4,
        \end{equation}
        where $\Lambda$ is the parameter of the spread-out lattice under consideration (if $\edges(\Z^d)$ is the nearest-neighbor edge set, we take $\Lambda = 0$ in the above). The reason for inclusion of this $\Lambda$ term has to do with a sequence of annuli $(Ann_i^q)$ defined at \eqref{eq:subann}. We wish to ensure that these annuli have pairwise disjoint boundaries; in particular, if $\cC$ is a spanning set of some $Ann_i^q$, it is not a spanning set of any other.
        
        The size of $k_1$ will determine the size of error estimates, in the following sense. In the final stages of our proof, we will show that 
        \begin{equation}
            \label{eq:firstsimp}
            \limsup_n \prob( E \mid 0 \sa{\Z^d \setminus D_n} V_n ) - \liminf_n \prob( E \mid 0 \sa{\Z^d \setminus D_n} V_n ) \leq \epsilon
        \end{equation}
        for some $\epsilon = \epsilon(k_1) \stackrel{k_1 \to \infty}{\longrightarrow} 0$; by taking $k_1 \to \infty$, we will establish \eqref{eq:thm1} and thereby prove Theorem~\ref{thm:main}. A similar error estimate applies to \eqref{eq:thm2}.
        
	We define the following {\bf scale sequences} $(k_i), (k_i^*)$ inductively, setting $k_i^* = 2 d^2 k_i $ and $k_i = 2 d^2 k_{i-1}^*$. Each of the two sequences grows  exponentially in $i$.
    We define a sequence of annuli $(Ann_i)_{i=1}^\infty$ by setting
    \begin{equation}
        \label{eq:annidef}
        Ann_i := B(2^{k_i^*+ 32d^4 + 1}) \setminus B(2^{k_i - 32d^4 - 1})\quad \text{for each $i \geq 1$}. 
    \end{equation}
    Since connections will be built from the origin $0$, it will simplify our notation to allow $0$ to be contained in some element of $(Ann_i)$. We therefore set $k_0^* = 1$ and let $Ann_0 := B(2) = B(2^{k_0^*})$ be an annulus in the degenerate sense of our definition. We append $Ann_0$ to the beginning of our sequence $(Ann_i)$.

	\subsubsection{Restricting to finite boxes\label{sec:supercrit}}
    The proofs of \eqref{eq:thm1} and \eqref{eq:thm2} are based on many of the same intermediate results about connectivity, and this common framework allows us to avoid excessive repetition of arguments. However, the presence of an infinite open cluster complicates the situation for $p > p_c$ --- for instance, $\prob_p(0 \lra x)$ does not decay to $0$ as $|x| \to \infty$ in this setting. 
    Our solution to this problem is to localize many of the connection events we consider.

    For each $p \in [p_c, 1)$, we set
    \begin{equation}
 	\label{eq:maxscale2}
 	\cutp = \cutp(p) := \sup \left\{ i: \frac{1}{2} \leq \frac{\prob_p(E)}{\prob_{p_c}(E)} \leq 2  \quad \text{for all $E \in \sigma\left((\omega_e)_{e \in B(2^{k_{i + 1}^*})}\right)$}\right\}\ .
 	\end{equation}
    In other words, this is the largest $i$ such that the Radon-Nikodym derivative $\mathrm{d}\prob_p \restriction_{\Sigma_{i + 1}} / \mathrm{d} \prob_{p_c} \restriction_{\Sigma_{i + 1}}$ of the restricted measures lies uniformly in $[1/2, 2]$, where $\Sigma_{i + 1}$ is the sigma-algebra generated by all bonds with both endpoints in $B(2^{k_{i+1}^*})$.
    We note that trivially, $\cutp(p) \to \infty$ as $p \searrow p_c$.



\subsection{\label{sec:regu} Regularity}
In \cite{KN}, Kozma-Nachmias introduced a powerful notion of regularity for high-dimensional open clusters. Their notions were extended in \cite{CH}. To give a sense of the use of their definition, note that the asymptotic \eqref{KN} shows there exists a $c > 0$ such that, for all $n \geq 1$,
\[\prob(0 \lra \partial B(2n) \mid 0 \lra \partial B(n)) \geq c\ .  \]
The arguments of \cite{KN} in fact show that a stronger conditional statement holds. Namely, such extensions are likely when the cluster $\fC_{B(n)}(0)$ is suitably regular: for all ``nice enough'' subgraphs $\cC \subseteq B(n)$, 
\[\prob(0 \lra \partial B(2n) \mid \fC_{B(n)}(0) = \cC) \geq c\ ,   \]
and it turns out to be typical that $\fC_{B(n)}(0)$ is nice in this sense.

Our arguments will require their own modified notion of regularity, where we consider clusters only within an element of our sequence $(Ann_i)_i$ from \eqref{eq:annidef} and so we introduce it with some care. In particular, we need to introduce sub-annuli of our main sequence $(Ann_i)$.  For each $i \geq 1$, we set 
\begin{equation}
    \label{eq:subann}
    Ann_i^{q} = B(2^{k_i^* + q}) \setminus B(2^{k_i - q}),\quad 0 \leq q \leq 32d^4\ .
\end{equation}
We also need to introduce an appropriate intermediate scale between the inner and outer radius of $Ann_i$. We set 
\begin{equation}
    \label{eq:lidef}
    \ell_i :=  d k_i \quad \text{and} \quad S_i := B\left(2^{\ell_i}\right)\ . 
\end{equation}

We also need an event describing a cluster not being excessively large.
For each $x \in \Z^d$ and integer $s \geq 1$, we define the event
\begin{equation}
    \label{eq:Tdef}
    \mathcal{T}_s(x):=\{|\fC(x) \cap B(x;s)| < s^4 \log^7 s\}.
\end{equation}
We use the event $\mathcal{T}_s(x)$ to define annular notions of regularity and badness. 
\begin{definition} \label{defin:regdef}
Suppose $x \in Ann_{i}^q$ for some $i \geq 1$ and $0 \leq q \leq 32d^4$.
\begin{enumerate}[label=\roman*]
    \item   We say $x$ is $(i, q,s)$-bad for an outcome $\omega$ if 
\begin{equation}
\label{eq:badness}
   \text{  in the outcome $\omega$, \quad } \mathbb{P}_{p_c}(\mathcal{T}_s(x)\mid\fC_{Ann_i^q}(x)) < 1 - \exp(-\log^2 s). 
\end{equation}
\item We say $x$ is $(i, q,K)$-regular  for an outcome $\omega$ if it is not $(i, q,s)$-bad for any $s \geq K$, and $x$ is $(i, q, K)$-irregular if it is not $(i, q,K)$-regular.
\item Finally, we define the random set
\begin{equation}
\label{eq:regset}
    Reg^{(i, q,K)}(x):=\{y \in \fC_{Ann_i^q}(x): y\,\text{is $(i, q,K)$-regular}\}.
\end{equation}
We note that according to our definitions 
\[\partial_{in} Reg^{(i, q,K)}(x) = Reg^{(i, q,K)}(x)\cap \partial_{in} \fC_{Ann_i^q}(x)\]
(suppressing the obvious dependence of the boundary on $q$ and $i$) and prefer the former notation for brevity, with a similar convention for the outer boundary $\partial_{out} Reg^{(i, q,K)}(x)$.
\end{enumerate}
\end{definition}

The notions introduced in Definition~\ref{defin:Kgood} refer to probabilities only at the critical point $p_c$ and not those for $p > p_c$. However, badness and regularity are properties of an outcome $\omega$, and so it makes sense to consider, for instance, the bad sites in an $\omega$ sampled from $\prob_p$ for $p > p_c$. 
We note that
\begin{equation}
    \label{eq:regcons} 
    \text{in $\omega$, if $x$ is not $(i, q, s)$-bad, then $|\fC_{Ann_i^q}(x)  \cap B(x;s)| < s^4 \log^7 s$},
\end{equation}
a fact which follows immediately from the definition. 
Importantly, $(i, q, K)$-regularity is a typical property: when $\fC_{Ann_i^q}(x)$ is large, then most of its vertices on the boundary of $Ann_i^q$ are $(i, q, K)$-regular. 

With these definitions and observations in hand, we can state a lemma that will be useful in the arguments in Section~\ref{sec:glue} below. It formalizes the notion that ``regularity is typical'' in the way that is most suitable for our geometry. 
\begin{lemma}
\label{lem:reg}
There exists a constant $K_0 > 1$ such that, for all $K \geq K_0$, there are constants $C, c > 0$ such that the following holds. For all choices of $k_1$ as in \eqref{eq:k0choice}, for any $i \geq 1$, all $0 \leq q \leq 32d^4$, and all $x \in Ann_i^0$, and for each $M\ge 1$: 
\begin{equation*}
        \mathbb{P}\left(\left|\partial_{in} \fC_{Ann_i^q}(x)\right| \geq M \,\text{and}\, \left|\partial_{in} Reg^{(i, q,K)}(x)\right| \leq \frac{\left|\partial_{in} \fC_{Ann_i^q}(x)\right|}{2}\right) \leq C   \exp\left(C k_i -c \log^2 M\right).
\end{equation*}
\begin{equation*}
\mathbb{P}\left(\left|\partial_{out} \fC_{Ann_i^q}(x)\right| \geq M \,\text{and}\, \left|\partial_{out} Reg^{(i, q,K)}(x)\right| \leq \frac{\left|\partial_{out} \fC_{Ann_i^q}(x)\right|}{2}\right) \leq C \exp(C k_i^* -c \log^2 M). 
\end{equation*}
\end{lemma}

Lemma~\ref{lem:reg} is proved by essentially the same method as in \cite{CH}; see also \cite{CHS} for similar adaptations of such regularity lemmas. As such, we omit a formal proof here.

	\section{Gluing constructions and lemmas\label{sec:glue}}
    \label{sec: structural}
This section is devoted to proving connectivity properties of large-scale clusters. In this introductory portion, we define (in Definition~\ref{defin:Kgood}) the notion of ``good'' sets with properties of typical large open clusters. We then (in Section~\ref{sec:subglue}) state and prove a lemma (Lemma~\ref{lem:glue}) arguing that when a large open cluster is a good set, its probability to have a further connection to another good spanning cluster or vertex is well-controlled. Finally, in Section~\ref{sec:tower}, we prove that when $0$ has an open connection to $V_n$ avoiding $D_n$ as in the conditioning from our Theorems~\ref{thm:main} and \ref{thm:main2}, this connection can be decomposed into a sequence of good sets and open paths between them.  

We recall  here the sequences of annuli we introduced at \eqref{eq:annidef} and \eqref{eq:subann}:
\begin{equation*}
        Ann_i := B(2^{k_i^*+ 32d^4 + 1}) \setminus B(2^{k_i - 32d^4 - 1})\quad \text{for each $i \geq 1$},
    \end{equation*}
and
\begin{equation*}
    Ann_i^{q} = B(2^{k_i^* + q}) \setminus B(2^{k_i - q}),\quad 0 \leq q \leq 32d^4\ .
\end{equation*}

	We further recall the parameter $K_0$ introduced in the statement of Lemma~\ref{lem:reg}. We will consider possible realizations of spanning clusters in annuli of the form $Ann_i^q$. The next definition, Definition~\ref{defin:Kgood}, describes the vertex sets which represent typical such realizations. 
    We recall, as explained just below \eqref{eq:k0choice}, that by our choice of scale sequences, $\partial_{in} Ann_i^q \cap \partial_{in} Ann_i^{r} = \varnothing$ for any $r \neq q$, with analogous statements holding for the $\partial_{out}$ boundaries as well. As such, if $\cC$ is a spanning set in $Ann_i^q$ for some $i$ and $q$, it cannot be a spanning set in any other values of $i$ or $q$. There is thus no ambiguity if we
    \begin{equation}
        \label{eq:Hdef}
        \text{define, for each spanning set $\cC$ in $Ann_i^q$,} \quad H(\cC) := \{\cC \text{ is a spanning cluster of $Ann_i^q$}\}\ ,
    \end{equation}
    since the annulus $Ann_i^q$ is uniquely determined by $\cC$. For similar reasons, we write $\partial_{in} \cC$ and $\partial_{out} \cC$ for $\partial_{in}^{Ann_i^q} \cC$ and $\partial_{out}^{Ann_i^q} \cC$ in Definition~\ref{defin:Kgood} and other contexts where it will not cause confusion. 
    
    As a final abbreviation in a similar spirit, for $\cC$ spanning $Ann_i^q$ as in the preceding discussion, we choose an $x \in \cC$ so that on $H(\cC)$ we have $\fC_{Ann_i^q}(x) = \cC$, and then write
    \begin{equation}
        \label{eq:Hdef2}
        Reg^K(\cC) := Reg^{(i, q,K)}(x)\ .
    \end{equation}
    (We note that $Reg^K(\cC)$ is constant on the event $H(\cC)$, taking the same value for each $\omega \in H(\cC)$.) Again, the values of $i$ and $q$ are uniquely determined by $\cC$, and so the omission of these indices does no harm.
    
	\begin{definition} \label{defin:Kgood}
		A subgraph $\cC$ of $Ann_i$ is a \emph{$K$-good spanning set} if:
	\begin{enumerate}
		\item $\cC$ is a connected spanning set of $Ann_i^q$ for some  (unique) $q \geq 1$.  \label{it:good1}
		\item Letting $1 \leq q \leq 32d^4$ be the index such that $\cC$ spans $Ann_i^q$, we have
        \[2^{7 k_i/4} \leq \left| \partial_{in} \cC \right| \leq 2^{9 k_i/4} \quad\text{and}\quad 2^{7 k_i^*/4} \leq \left| \partial_{out} \cC \right| \leq 2^{9 k_i^*/4}.\]
		\item  \label{it:good3} Letting $H(\cC)$ be as defined above at \eqref{eq:Hdef}, we have
        \begin{align*}
            H(\cC) \subseteq \left\{\left|\partial_{in} \cC \cap Reg^K(\cC)\right| \geq \frac{1}{2} \left|\partial_{in} \cC\right| \right\} \cap \left\{\left|\partial_{out} \cC \cap Reg^K(\cC)\right| \geq \frac{1}{2} \left|\partial_{out} \cC\right| \right\}\ ,
        \end{align*}
        recalling that at \eqref{eq:Hdef2} we introduced this abbreviation for the regular set from Definition~\ref{defin:regdef}. 
    \item \label{it:good4} For each $1 \leq r < q$, there is no connected component $\mathcal{D}$ of $\cC \cap Ann_i^r$ such that the previous items all hold (replacing $\cC$ by $\cD$ in each).
	\end{enumerate}
 When the value of $K$ is clear from context, we simply call $\cC$ a \emph{good spanning set}. 
	\end{definition}
 In fact, in most of what follows, we will fix a particular $K > K_0$ and omit the $K$-dependence from our notation. The specific value of $K$ will be chosen according to the statement of Lemma~\ref{lem:glue} below and will depend only on the lattice under consideration (in particular, not on the sequences $(D_n)$ or $(V_n)$, or on the choice of event $E$ from \eqref{eq:thm1} or \eqref{eq:thm2}).

    Definition~\ref{defin:Kgood} primarily exists to describe possible open clusters; when an annulus $Ann_i$ is spanned by an open cluster, we can typically find a good spanning set $\cC$ such that $H(\cC)$ occurs. The final item~\ref{it:good4} of Definition~\ref{defin:Kgood} is introduced to ensure that such $\cC$ is chosen ``minimally''. One could envision discovering such a set $\cC$ for which $H(\cC)$ occurs by exploring open clusters in $Ann_i^q$ for increasing values of $q$, and stopping at the first value of $q$ for which a $\cC$ satisfying the other items of Definition~\ref{defin:Kgood} may be found. A set~$\cC$ discovered by such a process would then satisfy item~\ref{it:good4} by construction.

    \subsection{\label{sec:subglue} Structural properties of large open clusters}
    Definition~\ref{defin:Kgood} will be most useful via an application of the next result, Lemma~\ref{lem:glue}, which shows that good spanning sets can be ``glued''. 
	For the statement of Lemma~\ref{lem:glue}, it will be useful to separate  each $Ann_i$ into the unequal ``halves'' $Ann_i \cap S_i$ and $Ann_i \setminus S_i$, where we recall the definition of $S_i$ and $\ell_i$ from \eqref{eq:lidef}. A connection from $\partial_{out} Ann_i$ to $V_n$ will be shown unlikely to enter the inner half $Ann_i \cap S_i$, along with a similar statement for connections from $0$ to $\partial_{in} Ann_i$.

	\begin{lemma} \label{lem:glue}
        There exist a choice of $K_2 > K_0$ and constants $0 < c < C < \infty$ depending only on the lattice such that the following holds.
		Uniformly in $k_1 > C$ used to define $(Ann_i)$ in \eqref{eq:k0choice},  in $p \in [p_c, 1)$, in indices $0 \leq i < j \leq \beta(p)$, in $K_2$-good spanning sets $\cC \subseteq Ann_i$ and $\cD \subseteq Ann_j$, and in vertices $x, y$ with $\|x\|_\infty \geq 2^{k_i^* + 2d + 6}$ and $\|y\|_\infty \leq 2^{k_i - 2d - 3}$, the following estimates hold: 
		\begin{enumerate}
			\item \label{it:twoclust} $\displaystyle c  |\partial_{out}\cC| |\partial_{in} \cD|  2^{(2-d)k_j}\leq \prob_p\big(\cC \sa{S_j \setminus S_i} \cD  \mid H(\cC), H(\cD) \big) \leq C  |\partial_{out}\cC| |\partial_{in} \cD| 2^{(2-d)k_j}\ ;$
			\item \label{it:knversion}  $ c |\partial_{out}\cC| |x|^{2-d} \leq \displaystyle \prob_p\left(\cC \xleftrightarrow{B(2\|x\|_{\infty}) \setminus S_i } x  \mid H(\cC) \right) \leq C |\partial_{out}\cC| |x|^{2-d};$
   \item \label{it:hs} $\displaystyle \prob_p\left(\cC \xleftrightarrow{B(\|x\|_{\infty})} x  \mid H(\cC)\right) \leq C |\partial_{out}\cC| |x|^{(1-d)};$
			\item \label{it:toorig} $c |\partial_{in}\cC| 2^{(2-d)k_i}\leq \displaystyle \prob_p\big(\cC \xleftrightarrow{S_i} y\mid H(\cC) \big) \leq C|\partial_{in}\cC| 2^{(2-d)k_i}. $
			\end{enumerate}
		\end{lemma}
    An estimate somewhat similar to a summed version of item \ref{it:knversion} from Lemma~\ref{lem:glue}, without the restriction to $B(2 \|x\|_\infty) \setminus S_i$, was a main estimate in \cite{KN}. Estimates of the probability of various restricted connectivity events similar to the ones in the statement of the lemma were proved in \cite{CH}. 
    
    However, the present setting, and in particular the content of item~\ref{it:twoclust} from Lemma~\ref{lem:glue}, differs substantially from the settings of previous work. We prove this item in detail,  then we give thorough sketches of the modifications necessary to prove the other parts of Lemma~\ref{lem:glue}.
  \begin{proof}[Proof of Lemma~\ref{lem:glue}, item~\ref{it:twoclust}]
    We note that the events $\left\{\cC \sa{S_j \setminus S_i} \cD\right\}$, $H(\cC)$, and $H(\cD)$ are cylinder events. By the definition of $\beta(p)$ in \eqref{eq:maxscale2}, it thus suffices to prove the lemma in case $p = p_c$. We therefore restrict to this case in all that follows.
In the course of our proof, we introduce two intermediate results. The first of these, Proposition~\ref{prop:Ybern}, is essentially graph-theoretic, and the second, Proposition~\ref{prop:Eprop}, involves diagrammatic estimates. To avoid interrupting the flow of the other arguments, we save the proof of these propositions for last.

The second inequality of item~\ref{it:twoclust} is relatively straightforward. By Lemma~\ref{lem:nofurther} and \eqref{2pt},
\begin{align}
    \nonumber \prob_{p_c}\Bigg(&\left\{\mathcal{C} \sa{S_j \setminus S_i} \mathcal{D}\right\} \cap H(\mathcal{C})\cap H(\mathcal{D})\Bigg) \\
     &\leq  \sum_{\substack{x \in \partial_{out}\cC \\ x^* \notin \cC: \, \{x, x^*\} \in \edges(\Z^d)}} \sum_{\substack{y \in \partial_{in}\cD \\ y^* \notin \cD: \, \{y, y^*\} \in \edges(\Z^d)}} \prob_{p_c}(x^* \xleftrightarrow{S_j \setminus [S_i \cup \cC \cup \cD]} y^*) \prob_{p_c}(H(\cC) \cap H(\cD))  \label{eq:indclust}\\
    \nonumber &\leq C \sum_{x \in \partial_{out}\cC} \sum_{y \in \partial_{in}\cD} \prob_{p_c}(x \leftrightarrow y)\prob_{p_c}(H(\cC) \cap H(\cD))\\
     &\leq \label{eq:useddist} C2^{(2-d)k_j} |\partial_{out}\cC| |\partial_{in} \cD| \prob_{p_c}(H(\mathcal{C})\cap H(\mathcal{D}))\ .
\end{align}
In \eqref{eq:indclust}, we used the independence of the $\omega_e$ variables corresponding to disjoint edge sets, and in \eqref{eq:useddist} we used that $\|x - y\|_\infty \geq 2^{k_j - 1}$ for all $x \in \cC$ and $y \in \cD$. This shows the upper bound on the probability appearing in item~\ref{it:twoclust} from Lemma~\ref{lem:glue}.

We now turn to the lower bound on the probability from item~\ref{it:twoclust}. We consider fixed values of $i$, $j$, and good spanning sets $\cC$, and $\cD$ in the arguments that follow. These values will be arbitrary, and all constants will be taken uniform relative to these choices, ensuring the uniformity claimed in the lemma's statement. As discussed below Definition~\ref{defin:Kgood}, $\cC$ is associated with $Ann_i^q$ for unique values of $i$ and $q$, with a similar statement for $\cD$. In what follows, then, we write $q_i$ for the unique choice of index such that $\cC$ is a spanning set of $Ann_i^{q_i}$ and write $q_j$ for the unique index such that $\cD$ is a spanning set of $Ann_j^{q_j}$.  We similarly use $i$ subscripts on other vertices associated to $Ann_i$ and $j$ subscripts on vertices associated to $Ann_j$ as a mnemonic device.

 For each vertex $x_i \in \partial_{out}\cC$ and each vertex $x_j \in \partial_{in} \cD$, we let $x_i^*$ (resp.~$x_j^*$) be a deterministically chosen neighbor of $x_i$ (resp.~$x_j$) satisfying $x_i^* \notin B(2^{k_i^* + q_i})$ (resp.~$x_j^* \in B(2^{k_j - q_j})$. 
 In other words, $x_i^*$ lies in the component of $\Z^d \setminus Ann_i^{q_i}$ which does not contain the origin, with an analogous statement holding for $x_j^*$.

With $i,$ $j$, $\cC$, and $\cD$ fixed as above, we define the random set 
\begin{equation}
    \label{eq:Ydef}
    Y(K) := \left\{(x_i, x_j): \, \begin{array}{c}
      x_i \in Reg^{K}(\cC) \cap  \partial_{out}\cC, \quad x_j \in Reg^{K}(\cD) \cap  \partial_{in}\cD: \\
      \{x_i, x_i^*\} \text{ and } \{x_j, x_j^*\} \text{ are open and pivotal for } \mathcal{C} \sa{S_j \setminus S_i} \mathcal{D}
    \end{array}\right\}\ .
\end{equation}
We will ultimately fix $K$ large independent of $i$, $j$, and the clusters $\cC$, $\cD$ such that the first moment of $|Y(K)|$ conditional on $H(\cC)\cap H(\cD)$ can be bounded uniformly below.  The role of this first moment is clarified by the following short proposition: 
\begin{proposition}\label{prop:Ybern}
For each $K \geq 1$,
$\prob_{p_c}(Y(K) \in \{0, 1\} \mid H(\cC)\cap H(\cD)) = 1$.
\end{proposition}
As already stated, we delay the proof of this fact until after the rest of our arguments are complete.
We will prove the lower bound of item~\ref{it:twoclust} from the lemma statement via the following estimate:
\begin{equation} \label{eq:vertunif}
\begin{split}
    \prob_{p_c}( \mathcal{C} \sa{S_j \setminus S_i} \mathcal{D} \mid H(\cC)\cap H(\cD)) &\geq \prob_{p_c}(|Y(K)| > 0\mid H(\cC)\cap H(\cD))\\
    &= \sum_{(x_i, x_j)} \prob_{p_c}((x_i, x_j) \in Y(K) \mid H(\cC)\cap H(\cD)).
\end{split}
\end{equation}
In the last step of \eqref{eq:vertunif}, we used Proposition~\ref{prop:Ybern}.  Item~\ref{it:twoclust} will follow immediately once we show
\begin{equation}\label{eq:vertunif2}
\begin{gathered}
    \text{there exists a $K_1 > K_0$ such that, for each $K \in (K_1,\, 2^{k_1 - 128d^4}) $, there exists a $c > 0$ with}\\
    \prob_{p_c}((x_i, x_j) \in Y(K) \mid H(\cC), H(\cD)) \geq c 2^{(2-d)k_j^*} \quad \text{uniformly in $x_i, x_j$}.
\end{gathered}
\end{equation}
(Of course $c$ is uniform in $i$, $j$, $\cC$, and $\cD$ as well). Inserting \eqref{eq:vertunif2} into \eqref{eq:vertunif} and summing over pairs concludes the proof of the lower bound of item~\ref{it:twoclust} from the lemma.

We show that \eqref{eq:vertunif2} holds by doing a modification argument related to those in~\cite{KN}. We define the event
\[\mathcal{E}_1 = H(\cC) \cap H(\cD)\ , \]
and for each collection of vertices $x_i,\, x_j,\, x_i',\, x_j'$ with $x_i \in Reg^{K}(\cC) \cap  \partial_{out}\cC$, with $x_j \in Reg^{K}(\cD) \cap  \partial_{in} \cD$, and with $x_i', x_j' \in B(2^{k_j - q_j}) \setminus B(2^{k_i + q_i})$, we define the events
\begin{equation}
\label{eq:Edef}
\begin{split}
    \mathcal{E}_2(x_i', x_j') &= \left\{ x'_i \xleftrightarrow{S_j \setminus[S_i \cup \cC \cup \cD]} x'_j\right\}, \\
    \mathcal{E}_3(x_i,x_j, x_i') &= \{\fC(x_i),\fC(x_j), \text{ and } \fC(x_i')\text{ are pairwise disjoint}\}\ .
    \end{split}
\end{equation}
In other words, $x_i'$ and $x_j'$ are in the annular region between the ones occupied by $\cC$ and $\cD$. In fact, they will be chosen as vertices near $x_i$ and $x_j$, respectively, namely lying in sets of the form
\begin{align*}
    \widehat B(z; K) := \left[B(2^{k_j - q_j}) \setminus B(2^{k_i + q_i})\right] \cap \left[B(z; K)\setminus B(z;\lfloor K/2\rfloor)\right] 
\end{align*}
for $z\in \{x_i,x_j\}$.

The main step toward establishing \eqref{eq:vertunif2} will be the bound from the next proposition, whose proof we also defer:
\begin{proposition}\label{prop:Eprop}
There exists a $K_1 > K_0$ such that, for all $K \in( K_1, \,  2^{k_1 - 128d^4})$, there is a $c = c(K) > 0$ such that the following holds for all $0 \leq i< j$. For each pair of good spanning sets $\cC$, $\cD$ as in the statement of the lemma, for each $x_i \in Reg^{K}(\cC) \cap  \partial_{out}\cC$ and $x_j \in Reg^{K}(\cD) \cap  \partial_{in} \cD$, there exist $x_i' \in \widehat B(x_i; K)$ and $x_j' \in \widehat B(x_j; K)$ with
    \begin{equation}\label{eqn: lwrbdE1E2E3}
     \mathbb{P}_{p_c}(\mathcal{E}_1 \cap \mathcal{E}_2(x_i, x_j) \cap \mathcal{E}_3(x_i, x_j, x_i')) \geq c 2^{(2-d) k_j} \prob_{p_c}(\mathcal{E}_1)\ .
\end{equation}
\end{proposition}
The presence of the vertices $x_i', x_j'$ is related to the reason for the assumed upper bound on $k_1$: we require at least a small amount of distance between vertices to make various error terms related to the triangle diagram small.
The argument that shows \eqref{eq:vertunif2}, and hence completes the proof of the lemma, involves a (by now somewhat standard) modification argument, which we briefly sketch now.

Considering an outcome $\omega \in \mathcal{E}_1 \cap \mathcal{E}_2(x_i, x_j) \cap \mathcal{E}_3(x_i, x_j, x_i')$, we note that in $\omega$, $\cC$ (resp.~$\cD$) is a subset of the open cluster $\fC(x_i)$ (resp.~$\fC(x_j)$). By the fact that $\omega \in \mathcal{E}_3(x_i, x_j, x_i')$, then, we have $\cC \not \lra \cD$ in $\omega$. By the fact that $\omega \in \mathcal{E}_2(x_i, x_j)$, there is an open path $\gamma$ from $x_i'$ to $x_j'$ with $\gamma \subseteq S_j \setminus \left[S_i \cup \fC(x_i) \cup \fC(x_j) \right]$.

We can therefore define a map 
\begin{equation}\label{eq:modifydef}
 \varphi: \, \mathcal{E}_1 \cap \mathcal{E}_2(x_i, x_j) \cap \mathcal{E}_3(x_i, x_j, x_i') \to \left\{(x_i, x_j) \in Y(K) \right\} \cap H(\cC) \cap H(\cD)
 \end{equation}
as follows. Letting $\gamma$ be as in the previous paragraph, the new outcome $\varphi(\omega)$ is produced from $\omega$ by closing all edges in $\widehat B(x_i; 2K)$ and $\widehat B(x_j; 2K)$ which do not intersect $\gamma$ and which are not open edges of $\cC$ or $\cD$. We then open edges along a deterministically chosen self-avoiding path from $x_i$ to $\gamma$ lying entirely in $\widehat B(x_i; 2K)$ and in a deterministically chosen path from  $x_j$ to $\gamma$ lying entirely in $\widehat B(x_j; 2K)$, with no edge of either path (except the first edge of each path) having an endpoint in $\cC$ or $\cD$.

By choosing the paths to end at their first intersections with $\gamma$, and to begin respectively with the edges $\{x_i, x_i^*\}$ and $\{x_j, x_j^*\}$, it follows that these two edges are open and pivotal for $\{\cC \sa{S_j \setminus S_i} \cD\}$ in $\varphi(\gamma)$. Because no edges adjacent to $\cC$ or $\cD$ except $\{x_i, x_i^*\}$ and $\{x_j, x_j^*\}$ are modified in this process, $\varphi(\omega) \in H(\cC) \cap H(\cD)$, and of course $x_i \in Reg^{(q_\cC, K)}(\cC)$ still holds (with a similar remark for $x_j$). In particular, $\varphi$ has the codomain claimed in its definition~\eqref{eq:modifydef}, and \eqref{eqn: lwrbdE1E2E3} will follow once we provide an appropriate lower bound on the probability of its image. 

On the other hand, $\varphi$ is at most $2^{|\edges(\widehat B(x_i; 2K))| + |\edges(\widehat B(x_i; 2K))|}$-to-one, and so we can lower bound the probability of its image using the multiple-valued map principle (see e.g.~\cite[Lemma 5.1]{DCGR}):
\[\prob_{p_c}\left(\varphi\left(\mathcal{E}_1 \cap \mathcal{E}_2(x_i, x_j) \cap \mathcal{E}_3(x_i, x_j, x_i')\right)\right) \geq c \prob_{p_c}\left(\mathcal{E}_1 \cap \mathcal{E}_2(x_i, x_j) \cap \mathcal{E}_3(x_i, x_j, x_i')\right)\ ,\]
where $c$ is uniform in all parameters except $K$. Applying \eqref{eqn: lwrbdE1E2E3} and recalling the codomain of $\varphi$ from \eqref{eq:modifydef} yields
\[\prob_{p_c}\left( \left\{(x_i, x_j) \in Y(K) \right\} \cap H(\cC) \cap H(\cD) \right)  \geq c 2^{(2-d) k_j} \prob_{p_c}(\mathcal{E}_1) = c 2^{(2-d) k_j} \prob_{p_c}(H(\cC) \cap H(\cD)),\]
completing the proof of \eqref{eq:vertunif2} and hence the proof of item~\ref{it:twoclust} from the lemma given Propositions~\ref{prop:Ybern} and \ref{prop:Eprop}. 
\end{proof}
We will now conclude the proof of the item by establishing these two propositions.

\begin{proof}[Proof of Proposition~\ref{prop:Ybern}]
We argue by contradiction, considering an outcome $\omega \in H(\cC) \cap H(\cD)$. We suppose that in $\omega$, the set $Y(K)$ contained two distinct pairs $(x_i, x_j)$ and $(y_i, y_j)$, and show that $\omega$ has contradictory properties. By the distinctness of the pairs under consideration, either $x_i \neq y_i$ or $x_j \neq y_j$. We argue that $\omega$ has contradictory properties in the case that $x_i \neq y_i$; the other case is similar.

In $\omega$, there exists a self-avoiding open path $\gamma\subseteq [S_j \setminus S_i]$ from $\cC$ to $\cD$, and by pivotality, it must contain both $\{x_i, x_i^*\}$ and $\{y_i, y_i^*\}$. Without loss of generality, we assume $\{x_i, x_i^*\}$ appears first in $\gamma$. We delete the initial segment of $\gamma$ up to, but not including, the appearance of the vertex $y_i$, producing a new open path $\gamma'$. Since $\gamma'$ is a subsegment of $\gamma$, it lies entirely in $S_j \setminus S_i$. In particular, in $\omega$, it witnesses the event $\{\mathcal{C} \sa{S_j \setminus S_i} \mathcal{D}\}$. 

However, $\{x_i, x_i^*\} \notin \gamma'$. This contradicts the definition of pivotality, showing that $\omega$ has contradictory properties as claimed. This completes the proof of item~\ref{it:twoclust} from Lemma~\ref{lem:glue} assuming the veracity of Propositions~\ref{prop:Ybern} and \ref{prop:Eprop}.
\end{proof}

The final ingredient is the somewhat lengthy proof of Proposition~\ref{prop:Eprop}.

\begin{proof}[Proof of Proposition~\ref{prop:Eprop}] 
In what follows, we frequently write ``for all $x_i$'' to mean ``for all $x_i \in Reg^{K}(\cC)$'' and similarly with $x_j$ to avoid excessive clutter; we will also say that bounds are simply ``uniform'' instead of recalling that the uniformity is over choices of $i$, $j$, $\cC$, $\cD$, and $(x_i, x_j)$. The parameter $K$ plays a special role, and so we continue to note $K$-dependence where applicable.

The first step toward proving Proposition~\ref{prop:Eprop} is the uniform lower bound
\begin{equation}
\begin{gathered}
    \text{there exists a $K' > K_0$ such that, for all $K \in (K', 2^{k_1-128d^4})$, there is a $c = c(K) > 0$}\\
   \text{such that} \quad \mathbb{P}_{p_c}(\mathcal{E}_1 \cap \mathcal{E}_2(x_i',x_j')) \geq c 2^{(2-d) k_j} \prob_{p_c}(\mathcal{E}_1)\ .
     \label{lem4.3}
\end{gathered}
 \end{equation}
To prove this, we first re-express the conditional probability:
\begin{align*}
    \mathbb{P}_{p_c}(\mathcal{E}_2(x_i',x_j') \mid \mathcal{E}_1) &= \mathbb{P}_{p_c}\left(x_i' \sa{S_j \setminus [S_i \cup \cC \cup \cD]} x_j' \mid \mathcal{E}_1\right)\\
    &= \mathbb{P}_{p_c}\left(x_i' \sa{S_j \setminus [S_i \cup \cC \cup \cD]} x_j'\right),
\end{align*}
since $\mathcal{E}_1$ depends only on the status of edges incident to $\cC$ and $\cD$, and the connection event does not depend on these edges.

We lower-bound the probability from the last display using a union bound:
\begin{align}
    \mathbb{P}_{p_c}(\mathcal{E}_2(x_i',x_j') \mid \mathcal{E}_1) &\geq \mathbb{P}_{p_c}(x'_i \leftrightarrow x'_j) - \prob_{p_c}(x_i' \lra x_j' \text{ via an open path intersecting } \cC \cup \cD \cup S_i \cup S_j^c)\nonumber\\
    \label{eq:outofbox}
    \begin{split}
    &\geq \mathbb{P}_{p_c}(x'_i \leftrightarrow x'_j) - \sum_{z \in \cC \cup \cD}\prob_{p_c}(\{x_i' \lra z\} \circ \{x_j' \lra z\})\\
    &\quad - \prob_{p_c}(\{x_i' \lra S_i\} \circ \{x_j' \lra S_i\}) - \sum_{z \in \partial S_j} \prob_{p_c}(\{x_i' \lra z\} \circ \{x_j' \lra z\})\ ,
    \end{split}
    \end{align}
    The BK inequality applied to the last inequality finally shows
    \begin{equation}
    \label{eq:E2terms}
    \begin{split}
    \mathbb{P}_{p_c}(\mathcal{E}_2(x_i',x_j') \mid \mathcal{E}_1)
    &\geq \mathbb{P}_{p_c}(x'_i \leftrightarrow x'_j) - \sum_{z \in \cC \cup \cD}\prob_{p_c}(x_i' \lra z) \prob_{p_c}(x_j' \lra z)\\
    &\quad - \prob_{p_c}(x_i' \lra S_i) \prob_{p_c}(x_j' \lra S_i) -  \sum_{z \in \partial S_j} \prob_{p_c}(x_i' \lra z) \prob_{p_c}(x_j' \lra z)\ .\end{split}
\end{equation}
We will therefore prove \eqref{lem4.3} by showing that the right-hand side of \eqref{eq:E2terms} is bounded below by $c 2^{(2-d) k_j}$ for some uniform constant $c > 0$.

We begin by lower-bounding the first term on the right side of \eqref{eq:E2terms}. Uniformly in $x_i'$ and $x_j'$, we have $\|x_i' - x_j'\|_\infty \leq 2^{k_j + 2}$. This ensures the existence of a uniform constant $c > 0$ such that
\begin{equation} \label{eq:E2terms1}
    \mathbb{P}_{p_c}(x'_i \leftrightarrow x'_j) \geq c 2^{(2-d) k_j}\ ,
\end{equation}
which is the order of the claimed lower bound for the entire right-hand side of \eqref{eq:E2terms}. We show that the negative terms are smaller in magnitude than the bound from \eqref{eq:E2terms1}, at least for appropriately chosen $K$ and appropriately chosen (as a function of $x_i,\, x_j$) vertices $x_i'$, $x_j'$.

We begin by bounding the first negative term of \eqref{eq:E2terms}, considering first the sum over $z \in \cC$. We note that as long as $K \leq 2^{k_j - 128d^4}$, such $z$ must satisfy $\|z - x_j'\|_{\infty} \geq 2^{k_j - 64 d^4 - 1}$. Using \eqref{2pt} shows that uniformly in such $K$ and $x_i, x_j, x_i', x_j'$,
\begin{align}
   \nonumber
    \sum_{z \in \cC}\prob_{p_c}(x_i' \lra z) \prob_{p_c}(x_j' \lra z) &\leq C 2^{(2-d) k_j} \sum_{z \in \cC}|x_i' - z|^{2-d} \\
    &\leq  C 2^{(2-d) k_j} \sum_{r = \lfloor\log_2 K-1\rfloor}^\infty 2^{(2-d)r} \left|\left\{z \in \cC: \, 2^{r} \leq \|z-x_i'\|_{\infty} < 2^{r+1} \right\}\right|\nonumber\\
    &\leq  C 2^{(2-d) k_j} \sum_{r = \lfloor\log_2 K - 1\rfloor}^\infty 2^{(2-d)r} \left|\left\{z \in \cC: \, \|z-x_i\|_{\infty} < 2^{r+4} \right\}\right|\ .  \label{eq:E2terms2a}
\end{align} 
In the last line, we used the fact that $\|x_i - x_i'\|_{\infty} \leq K$.
The sum on the right is bounded using the observation \eqref{eq:regcons}, that a regular vertex like $x_i$ cannot have $\fC_{Ann_i^{q_i}}(x_i)$ too large, yielding
\begin{equation}\label{eq:firstterm}
    \sum_{z \in \cC}\prob_{p_c}(x_i' \lra z) \prob_{p_c}(x_j' \lra z) \leq C 2^{(2-d) k_j} \sum_{r = \lfloor\log_2 K\rfloor}^\infty 2^{(2-d)r} [r^7 2^{4 r}] \leq C 2^{(2-d) k_j} K^{6-d} \log^7 K\ .
\end{equation}
A similar series of estimates shows an identical bound for the portion of the first negative term of \eqref{eq:E2terms} corresponding to the sum over $z \in \cD$. 

The quantity on the right-hand side of \eqref{eq:firstterm} will serve as our bound on the first negative term of \eqref{eq:E2terms}. The other terms are bounded using the definitions of the scale sequences from directly above \eqref{eq:annidef} and the estimate \eqref{2pt}, and here we crucially use the rapid growth of these scale sequences. Indeed, a union bound for the second negative term of \eqref{eq:E2terms} with $d \geq 7$ gives 
\begin{align}
     \prob_{p_c}(x_i' \lra S_i) \prob_{p_c}(x_j' \lra S_i) &\leq \sum_{z_1, z_2 \in S_i} \prob_{p_c}(x_i' \lra z_1) \prob_{p_c}(x_j' \lra z_2)\nonumber\\
    &\leq C |S_i|^2  2^{(2-d) k_i^*} 2^{(2-d) k_j} \leq C 2^{2d k_i}2^{(2-d) k_i^*} 2^{(2-d) k_j}\nonumber\\
    &\leq C 2^{(2d + 4d^2 - 2d^3) k_i} 2^{(2-d) k_j} \leq C 2^{-8 d^2 k_i} 2^{(2-d) k_j}\ .\label{eq:secondterm}
\end{align}
The final term of \eqref{eq:E2terms} satisfies
\begin{align}
\label{eq:thirdterm}
    \sum_{z \in \partial S_j} \prob_{p_c}(x_i' \lra z) \prob_{p_c}(x_j' \lra z) &\leq C 2^{(4-d) \ell_j} \leq C 2^{-3d k_j}\ . 
\end{align}

Combining \eqref{eq:E2terms1}, \eqref{eq:firstterm}, \eqref{eq:secondterm}, and \eqref{eq:thirdterm}, we see that the right-hand side of \eqref{eq:E2terms} is at least
\[ (c - C_1 K^{6-d} \log^7 K - C_2 2^{-2d k_1}) 2^{(2-d) k_j},\]
where we recall $k_1$ from the definition of the scale sequences above \eqref{eq:annidef}. The claimed estimate \eqref{lem4.3} is an immediate consequence of the last display.

It now remains to prove the proposition using \eqref{lem4.3}. We first arrive at an expression with error terms similar to \eqref{eq:E2terms}, then bound the resulting terms. If $\mathcal{E}_1 \cap \mathcal{E}_2(x_i, x_j)$ occurs but $\mathcal{E}_3(x_i, x_j, x_i')$ does not, then in particular either 1) there is an open connection from the set $\{x_i, x_j\}$ to $x_i'$, or 2) there is an open connection from $x_i$ to $x_j$ and no connection from $x_i$ to $x_i'$. In the former case, we can decompose the open connection based on which of $\cC$ and $\cD$ it intersects last. A union  bound yields
\begin{equation} \label{eq:E3terms}
\begin{split}
    \prob_{p_c}(\mathcal{E}_2(x_i', x_j') \cap &\mathcal{E}_3(x_i, x_j, x_i') \mid \mathcal{E}_1) \geq\\
    &\prob_{p_c}(\mathcal{E}_2(x_i', x_j') \mid \mathcal{E}_1)
    - \prob_{p_c}(\mathcal{E}_2(x_i', x_j') \cap \{x_i \sa{\Z^d \setminus \cD} x_i'\} \mid \mathcal{E}_1)\\
    -  ~&\prob_{p_c}(\mathcal{E}_2(x_i', x_j') \cap \{x_j \sa{\Z^d \setminus \cC} x_i'\} \mid \mathcal{E}_1)\\
    - ~&\prob_{p_c}(\mathcal{E}_2(x_i', x_j') \cap \{x_i \lra x_j\} \setminus \{x_i \lra x_i'\} \mid \mathcal{E}_1)\ .
    \end{split}
\end{equation}
Each of these negative terms may be represented diagrammatically. See Figures \ref{fig:one}, \ref{fig:two}, \ref{fig:three}, where connections in addition to those in $\mathcal{E}_2$ appear in red. 
The first two negative terms of \eqref{eq:E3terms} will be bounded using regularity in a typical way, while the last requires other properties from Definition~\ref{defin:Kgood}.

\begin{figure}[htbp]  
  \centering
  \includegraphics[width=0.65\textwidth]{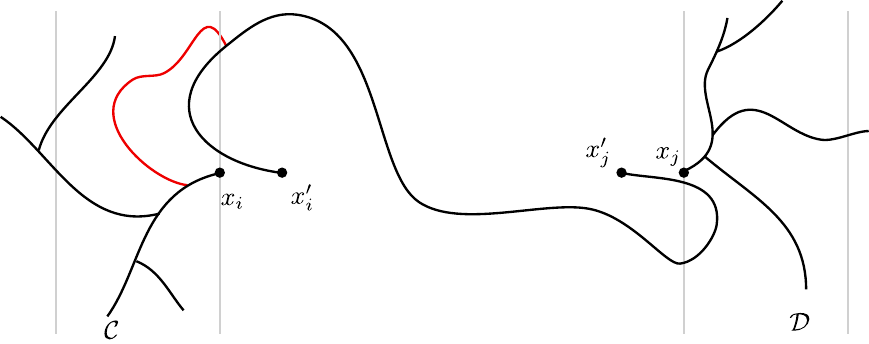}
  \caption{The event $\mathcal{E}_2(x_i', x_j') \cap \{x_i \sa{\Z^d \setminus \cD} x_i'\}$.}
  \label{fig:one}
\end{figure}

\begin{figure}[htbp]  
  \centering
  \includegraphics[width=0.65\textwidth]{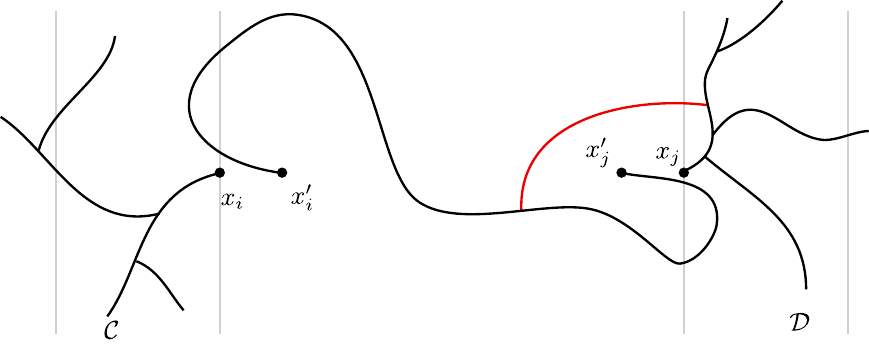}
  \caption{The event $\mathcal{E}_2(x_i', x_j') \cap \{x_j \sa{\Z^d \setminus \cC} x_i'\}$.}
  \label{fig:two}
\end{figure}

\begin{figure}[htbp]  
  \centering
  \includegraphics[width=0.65\textwidth]{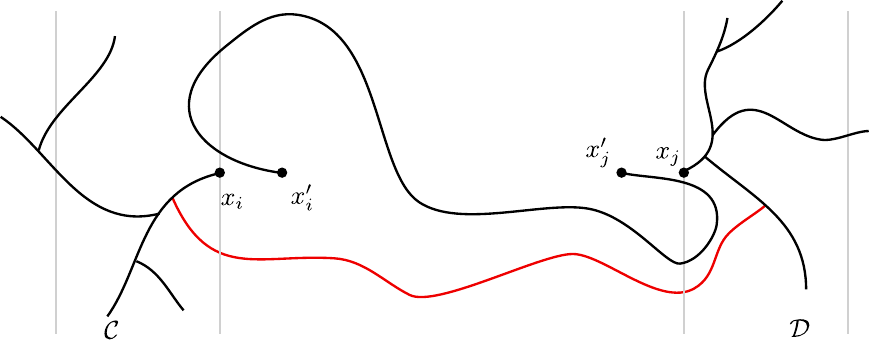}
  \caption{The event $\mathcal{E}_2(x_i', x_j') \cap \{x_i \lra x_j\} \setminus \{x_i \lra x_i'\}$.}
  \label{fig:three}
\end{figure}

We describe first the method for bounding the first negative term of \eqref{eq:E3terms}. When $\mathcal{E}_1 \cap \mathcal{E}_2(x_i, x_j) \cap \{x_i \sa{\Z^d \setminus \cD} x_i'\}$ occurs, there exists an open connection $\gamma$ from $x_i'$ to $x_j'$ avoiding $\cC \cup \cD$, and there is some open connection from $x_i$ to $x_i'$ which avoids $\cD$ and necessarily intersects $\gamma$. (The connection from $x_i$ to $x_i'$ may use open edges of $\cC$.) Choosing an appropriate vertex $z \in \gamma$, we find disjoint witnesses showing the following set inclusion:
\begin{equation} \label{eq:E3intermed}
    \mathcal{E}_1 \cap \mathcal{E}_2(x_i, x_j) \cap \{x_i \sa{\Z^d \setminus \cD} x_i'\} \subseteq \bigcup_{z} \left[\mathcal{E}_1 \cap \{z \sa{\Z^d \setminus \cD} x_i\}\right] \circ \{z \lra x_i'\} \circ \{z \lra x_j'\}  \ .
\end{equation}

Applying the BK inequality, we see the probability of the event on the left side of \eqref{eq:E3intermed} satisfies
\begin{equation*} 
    \prob_{p_c}\left(\mathcal{E}_1 \cap \mathcal{E}_2(x_i, x_j) \cap \{x_i \sa{\Z^d \setminus \cD} x_i'\}\right) \leq C \sum_z \prob_{p_c}\left( \mathcal{E}_1 \cap \{z \sa{\Z^d \setminus \cD} x_i\}\right) |z - x_i'|^{2-d} |z-x_j'|^{2-d}\ ,
\end{equation*}
and using independence of the events $\{x_i \sa{\Z^d \setminus \cD} x_i'\}$ and $H(\cD)$ yields
\begin{align}
     \nonumber\prob_{p_c}\left(\left.\mathcal{E}_2(x_i, x_j) \cap \{x_i \sa{\Z^d \setminus \cD} x_i'\} \right| \mathcal{E}_1\right) &\leq C\sum_z \prob_{p_c}\left(\left. z \sa{\Z^d \setminus \cD} x_i \right| \mathcal{E}_1 \right) |z - x_i'|^{2-d} |z-x_j'|^{2-d}\\
       \label{eq:E3intermed2} &\leq C\sum_z \prob_{p_c}\left(z \lra x_i \mid H(\cC)\right) |z - x_i'|^{2-d} |z-x_j'|^{2-d}\ .
\end{align}

The terms of \eqref{eq:E3intermed2} corresponding to $z \notin B(2^{k_j- 1})$ may be bounded by controlling the first two factors uniformly in $x$ and good spanning sets $\cC$. For $\{x_i \lra z\} \cap H(\cC)$ to occur, there must be some $y \in \partial \cC$ satisfying $y \sa{\Z^d \setminus \cC} z$; taking a union bound over such $y$ and using $\min\{\|y - z\|_\infty, \|x_i - z\|_\infty\} \geq \|z\|_\infty / 2$ yields
\begin{align}
    \sum_{z \notin B\left(2^{k_j -1}\right)} \prob_{p_c}\left(z \lra x_i \mid H(\cC)\right) |z - x_i'|^{2-d} |z-x_j'|^{2-d} \leq  \left[\left|\partial_{out}\mathcal{C} \right| + \left| \partial_{in}\mathcal{C} \right| \right]  \sum_{z\notin B\left(2^{k_j -1}\right)} |z|^{4-2d}|z-x_j'|^{2-d}\ ,\label{eq:nearz0}
    \end{align}
    and then uniformly bounding $|z|^{2-d} \leq C 2^{(2-d) k_j}$, and using the cardinality bounds on $\partial_{out}\mathcal{C}, \, \partial_{in}\mathcal{C}$ from Definition~\ref{defin:Kgood}  gives
    \begin{align}
           \eqref{eq:nearz0}
           &\leq C  2^{9 k_i^*/4 - (2-d) k_j}\left[ \sum_{z \notin B(2^{k_j - 1})}|z|^{2-d}|z-x_j'|^{2-d}  \right] \\
   &\leq C2^{9 k_i^*/4 - (6-2d) k_j} \leq C2^{(1-d)k_j}\ , \label{eq:nearz}
\end{align} 
where in the last line we used \eqref{eqn: convolution} and the definition of the scale sequences above \eqref{eq:annidef}. 

To control the terms of \eqref{eq:E3intermed2} when $z \in B(2^{k_j- 1})$, we bound $\|z-x_j'\|_{\infty} \geq 2^{k_j-1}$ to see
\begin{align}\label{eq:torand}
   \sum_{z \in B\left(2^{k_j -1}\right)} \prob_{p_c}\left(z \lra x_i \mid H(\cC)\right) |z - x_i'|^{2-d} |z-x_j'|^{2-d} &\leq C 2^{(2-d)k_j}  \sum_{z} \prob_{p_c}\left(z \lra x_i \mid H(\cC)\right) |z - x_i'|^{2-d}\ .
\end{align}
The remaining sum can be bounded by similar means to those in \cite{KN} by a further careful decomposition over scales. The portion of the sum on the right-hand side of \eqref{eq:torand} corresponding to $z \notin B(x_i; 2K)$ is simpler, using Definition~\ref{defin:regdef}:
\begin{align}
    \sum_{z \notin B(x_i; 2K)} \prob_{p_c}\left(z \lra x_i \mid H(\cC)\right) |z - x_i'|^{2-d} &\leq C \sum_{r = \lfloor \log_2 K \rfloor}^\infty 2^{(2-d)r} \mathbb{E}\left[ \left|\fC(x_i)\cap B(x_i,2^{r}) \right|\  \mid H(\cC)\right]\nonumber\\
    &\leq C \sum_{r = \lfloor \log_2 K \rfloor}^\infty 2^{(6-d)r}r^7 \leq C K^{6-d} \log^7 K\ .\label{eq:farz}
\end{align}

For the other terms of \eqref{eq:torand}, we sum further over $x_i' \in \widehat B(x_i; K)$:
\begin{align}
    \sum_{x_i' \in \widehat B(x_i; K)}\sum_{\substack{z \in B(x_i; 2K)}} \prob_{p_c}\left(z \lra x_i \mid H(\cC)\right) |z - x_i'|^{2-d} &\leq C K^2 \sum_{\substack{z \in B(x_i; 2K)}} \prob_{p_c}\left(z \lra x_i \mid H(\cC)\right) \nonumber\\
    &\leq  C K^6 \log^7 K\ .\label{eq:rand2}
\end{align}
In particular, the right-hand side of \eqref{eq:rand2} implies 
\begin{equation}
    \label{eq:rand3}
    \begin{gathered}
    \text{there is an $x_i' \in \widehat B(x_i; K)$ such that }\\
    \sum_{\substack{z \in B(x_i; 2K)}} \prob_{p_c}\left(z \lra x_i \mid H(\cC)\right) |z - x_i'|^{2-d} \leq C K^{6-d} \log^7 K\ .
    \end{gathered}
\end{equation}

Inserting \eqref{eq:farz} and \eqref{eq:rand3} into \eqref{eq:torand}, and then combining the result with \eqref{eq:nearz} allows us to bound the left-hand side of \eqref{eq:E3intermed2}. This provides an upper bound for the first negative term of \eqref{eq:E3terms}:
\begin{equation}
    \label{eq:E3intermedbd}
    \begin{gathered}
    \text{for each $K \in (K_0, 2^{k_1-128d^4})$, there exists $x_i' \in \widehat B(x_i; K)$ satisfying}\\
    \prob_{p_c}\left( \left. \mathcal{E}_2(x_i', x_j') \cap \left\{x_i \sa{\Z^d \setminus \cD} x_i'\right\}\right| \mathcal{E}_1\right) \leq C 2^{(2-d)k_j} K^{6-d} \log^7 K\ .
    \end{gathered}
\end{equation}
As usual, the constant $C$ is uniform in $K$ and in the other parameters $i$, $j$, $x_i, x_j, x_j'$. 

The second negative term can be bounded via a similar series of estimates. Because the argument is essentially the same as the one leading to \eqref{eq:E3intermedbd}, we do not give the details; we record the result here:
\begin{equation}
 \label{eq:E3intermedbd2}
    \begin{gathered}
    \text{for each $K \in (K_0, 2^{k_1-128d^4})$, there exists $x_j' \in \widehat B(x_j; K)$ satisfying}\\
    \prob_{p_c}\left(\left.\mathcal{E}_2(x_i, x_j) \cap \left\{x_j \sa{\Z^d \setminus \cC} x_i'\right\} \right| \mathcal{E}_1\right) \leq C 2^{(2-d)k_j} K^{6-d} \log^7 K\ ,
    \end{gathered}
\end{equation}
uniformly in $K$ and in the other parameters $i$, $j$, $x_i, x_j, x_i'$.

We finally bound the last negative term of \eqref{eq:E3terms}. Given an outcome in the event
\begin{equation}
    \label{eq:lastevent}
    \mathcal{E}_2(x_i', x_j') \cap \left(\{x_i \lra x_j\} \setminus \{x_i \lra x_i'\}\right) \cap \mathcal{E}_1\ ,
\end{equation} 
we can find edge-disjoint witnesses for the three events 
\[\mathcal{E}_1,\quad \{x_i' \lra x_j'\}, \quad \text{and}\quad \{\exists \text{ an open path from a vertex of $\cC$ to a vertex of $\cD$}\}.\]
In particular, using the BK-Reimer inequality and Lemma~\ref{lem:nofurther}, we see that
\begin{align}
    \prob_{p_c}(\eqref{eq:lastevent}) &\leq C \prob_{p_c}(\mathcal{E}_1) 2^{(2-d)k_j} \prob_{p_c}\left(\left[ \partial_{in} \cC \cup \partial_{out} \cC \right] \lra \left[ \partial_{in} \cD \cup \partial_{out} \cD \right] \right)\nonumber\\
    &\leq C \prob_{p_c}(\mathcal{E}_1) 2^{(2-d)k_j} \left[\left|\partial_{in} \cC\right|+ \left|\partial_{out} \cC\right| \right] \left[2^{(2-d) k_j^*}\left|\partial_{in} \cD\right| +\left|\partial_{out} \cD\right| 2^{(2-d) k_j}\right]\nonumber\\
    &\leq C \prob_{p_c}(\mathcal{E}_1) 2^{(2-d)k_j} \left[2^{(9/4)k_i^*} \right] \left[2^{(9/4 + 2 - d) k_j}+ 2^{(9/4 + 2 - d) k_j^*} \right]\nonumber\\
    &\leq C \prob_{p_c}(\mathcal{E}_1) 2^{2(2-d)k_j} \left[ 2^{(9/4)k_i^* } \right] \left[2^{(9/4) k_j} \right]\nonumber\\
    &\leq C \prob_{p_c}(\mathcal{E}_1) 2^{(13/2 -d)k_j}2^{(2 -d)k_j}\ .\label{eq:lastfrac}
\end{align}
In the last line, we used the trivial inequality $k_i^* \leq k_j$. 

Dividing both sides of \eqref{eq:lastfrac} by $\prob_{p_c}(\mathcal{E}_1)$ and noting that $d \geq 7$, we find the following bound for the third negative term of \eqref{eq:E3terms}:
\begin{equation}\label{eq:E3intermedbd3}
    \begin{gathered}
    \text{for each $K \in (K_0, 2^{k_1-128d^4})$ and each pair $x_i'$ and $x_j'$, we have}\\
    \prob_{p_c}(\mathcal{E}_2(x_i, x_j) \cap \{x_i \lra x_j\} \setminus \{x_i \lra x_i'\} \mid \mathcal{E}_1) \leq C 2^{(3/2 -d)k_j}\ .
    \end{gathered}
\end{equation}
With all these estimates accomplished, it remains only to gather the resulting bounds. 

Applying \eqref{lem4.3} to lower-bound the positive term of \eqref{eq:E3terms}, then using \eqref{eq:E3intermedbd}, \eqref{eq:E3intermedbd2}, and \eqref{eq:E3intermedbd3} to control the negative ones, we see there exist constants $c, C > 0$ and $K' > K_0$ such that, for all $K \in (K', 2^{k_1-128d^4})$, there exist choices of $x_i'$ and $x_j'$ such that, uniformly in all other parameters, we have
\begin{align*}
    \prob_{p_c}(\mathcal{E}_2(x_i, x_j) \cap \mathcal{E}_3(x_i, x_j, x_i') \mid \mathcal{E}_1) \geq c 2^{(2-d)k_j}\left(1 - C K^{6-d} \log^7 K - C 2^{-k_1/2} \right)\ .
\end{align*}
Noting that the quantity in parentheses is uniformly (in $i, j, x_i, x_j$) positive for all large $K$ and $k_1$ completes the proof of the proposition. 
\end{proof}

We have now proved the first item from the statement of Lemma~\ref{lem:glue} in detail. As promised, we sketch the modifications to the above proof which are necessary to prove the remaining items.
\begin{proof}[Sketch of proof of items \ref{it:knversion} -- \ref{it:toorig} from Lemma~\ref{lem:glue}]
    We describe the necessary changes to the above proof in each case, proceeding in order.
    
    {\bf Item~\ref{it:knversion}.} Here we follow a very similar, but simpler line of argument to the one given in the case of Item~\ref{it:twoclust}. We begin again with the proof of the upper bound. The line of reasoning that led to \eqref{eq:useddist} produces an analogous expression here:
    \begin{align*}
        \prob_{p_c}\left(\left\{\cC \xleftrightarrow{B(2\|x\|_{\infty}) \setminus S_i } x\right\} \cap H(\cC) \right)
     &\leq \prob_{p_c}(\{\cC \lra x\}  \cap H(\cC) )\\
      &\leq C|x|^{2-d} |\partial_{out}\cC| \prob_{p_c}(H(\mathcal{C}))\ .
    \end{align*} 
    The chief difference, as visible in the last line of the above display, is that no union bound over the set $\cD$ is necessary. The upper bound of item~\ref{it:knversion} follows.

    The lower bound of item~\ref{it:knversion} uses a modification argument very close to that of item~\ref{it:twoclust}. 
    In particular, the goal is to bound the probability of the event that the following random set is nonempty:
    \[ Y(K) := \left\{x_i: \, 
      x_i \in Reg^{K}(\cC) \cap  \partial_{out} \cC,\ 
      \{x_i, x_i^*\}  \text{ is open and pivotal for } \mathcal{C}\xleftrightarrow{B(2\|x\|_{\infty}) \setminus S_i } x
    \right\}\]
    There is no second vertex $x_j$ involved; this difference is reflected in the modified definition of the $\mathcal{E}_i$ events. Now we simply set $\mathcal{E}_1 = H(\cC)$ and, letting $x_i'$ denote an arbitrary vertex of $\widehat B(x_i; K)$ as before,
    \begin{align*}
        \mathcal{E}_2(x_i') = \left\{ x'_i \xleftrightarrow{B(2\|x\|_{\infty}) \setminus S_i } x\right\}, \quad
        \mathcal{E}_3(x_i, x_i') = \{\fC(x_i) \text{ and } \fC(x_i')\text{ are disjoint}\}\ .
    \end{align*}

    We proceed by proving estimates on the probabilities of these events. The first step is the following analogue of \eqref{lem4.3}:
    \begin{equation*}
    \begin{gathered}
    \text{there exists a $K' > K_0$ such that, for all $K > K'$, there is a $c = c(K) > 0$}\\
   \text{such that for all $x_i' \in \widehat B(x_i; K)$,} \quad \prob_{p_c}(\mathcal{E}_1 \cap \mathcal{E}_2(x_i')) \geq c |x|^{2-d} \prob_{p_c}(\mathcal{E}_1)\ .
\end{gathered}
 \end{equation*}
    The argument for this estimate is similar to that for \eqref{lem4.3} --- or indeed to the argument in \cite{KN}, but now using the restricted two-point estimate \eqref{eq:restrtwopt}.

    The final main step is a union bound similar to \eqref{eq:E3terms}:
    \begin{align*}
         \prob_{p_c}(\mathcal{E}_2(x_i') \cap \mathcal{E}_3(x_i, x_i') \mid \mathcal{E}_1) \geq
    \prob_{p_c}(\mathcal{E}_2(x_i') \mid \mathcal{E}_1)
    - \prob_{p_c}((\mathcal{E}_2(x_i') \cap \{x_i \lra x_i'\} \mid \mathcal{E}_1)\ .
    \end{align*}
    The negative term of the last display is similar to the first negative term of \eqref{eq:E3terms}, and we do not require the other two negative terms of \eqref{eq:E3terms} because there is no other cluster $\cD$ under consideration. The negative term of the above display can be estimated similarly to the methods which produced \eqref{eq:E3intermedbd}, leading to the analogous bound
    \begin{equation*}
    \begin{gathered}
    \text{for each $K$, and for an appropriate $x_i' \in \widehat B(x_i; K)$,}\\
    \prob_{p_c}(\mathcal{E}_2(x_i') \cap \{x_i \lra x_i'\}\mid \mathcal{E}_1) \leq C |x|^{2-d} K^{6-d} \log^7 K\ .
    \end{gathered}
\end{equation*}

    Pulling these estimates together, we find the lower bound
    \[ \prob_{p_c}(\mathcal{E}_2(x_i') \cap \mathcal{E}_3(x_i, x_i') \mid \mathcal{E}_1) \geq c |x|^{2-d}\ , \]
    for an appropriate choice of $x_i'$ for each $x_i$. A modification estimate using a map similar to the one defined in \eqref{eq:modifydef} then shows that
    \[\mathbb{E}[Y(K) \mid \mathcal{E}_1] \geq c |\partial_{out}\cC| |x|^{2-d}\ , \]
    and then the proof is completed using the analogue of \eqref{eq:vertunif}, which holds for essentially the same reasons as before. Item~\ref{it:knversion} of the lemma is proved.

    {\bf Proof of item~\ref{it:hs}.} We again follow the proof that led to \eqref{eq:useddist}. This leads to the following sequence of estimates: 
    \begin{align*}
        \prob_{p_c}(\cC \sa{B(\|x\|_{\infty})} x  \mid H(\cC) ) 
      &\leq \left[|\partial_{out}\cC| + |\partial_{in}\cC| \right] \sup_{y \in \partial_{in} \cC \cup \partial_{out}\cC} \prob_{p_c}(y \sa{B(\|x\|_{\infty})} x )\\
      &\leq C |x|^{1-d} |\partial_{out}\cC| \prob_{p_c}(H(\mathcal{C}))\ .
    \end{align*} 
    In the last line, we used the fact that the connection from $y \in \cC$ to $x$ must lie in an appropriate half-space. For instance, if $x$  satisfies $x \cdot \mathbf{e}_1 = \|x\|_\infty$, then the connection from $y$ to $x$ in particular lies in
    \[\{z \in \Z^d: \, z \cdot \mathbf{e}_1 \leq \|x\|_\infty\}\ . \]
    We then applied the half-space two-point bound \eqref{eq:hstwopt}. This completes the proof.

    {\bf Proof of item~\ref{it:toorig}.}  The argument is quite similar to the one used to show item~\ref{it:knversion}. The upper bound again is essentially the same argument as used in \eqref{eq:useddist}, noting in this case that the connection from $\cC$ to $y$ cannot use a vertex of $\partial_{out} \cC$; such a connection would intersect $\partial S_i$, which is forbidden.

    The proof of the lower bound from item~\ref{it:toorig} follows the outline of the argument for item~\ref{it:knversion}. The chief difference comes in the proof of the analogue of \eqref{lem4.3}. We adapt the definition \eqref{eq:Edef} to the present circumstances by setting, for $x_i'$ near a typical vertex $x_i \in Reg^{K}(\cC) \cap \partial_{in}\cC$, 
    \[\mathcal{E}_2(x_i') = \left\{ x'_i \sa{\Z^d \setminus [\cC \cup \partial S_i]} y\right\}\ . \]
    Arguing to produce a lower bound like \eqref{eq:outofbox}, we arrive at
    \begin{align*}
    \prob_{p_c}(\mathcal{E}_2(x_i') \mid H(\cC)) &\geq \prob_{p_c}(x'_i \leftrightarrow y) - \prob_{p_c}(x_i' \lra y \text{ via an open path intersecting } \cC \cup \partial S_i)\\
    &\geq \prob_{p_c}(x_i' \leftrightarrow y) - \sum_{z \in \cC}\prob_{p_c}(\{x_i' \lra z\} \circ \{x_j' \lra z\})\\
    &\quad - \sum_{w \in \partial S_i} \prob_{p_c}(\{x_i' \sa{S_i} w\} \circ \{w \lra y\})\ .
    \end{align*}

    The negative terms from the last display are controlled like the first and third negative terms of \eqref{eq:E2terms}. The first term of the display is lower-bounded using \eqref{2pt}. The remainder of the proof of the item is quite similar to the remainder of the proof of item~\ref{it:knversion}, so we refer the reader to the sketch of that item for more details.
\end{proof}

We introduce the following convention throughout the remainder of the paper:
\begin{equation}\begin{gathered}
        \text{\bf In all that follows, we choose $K = K_2$ for the constant $K_2$ of Lemma~\ref{lem:glue}}
        \\ \text{whenever the notions of Definition~\ref{defin:regdef} are referenced. }
    \label{eq:Kfixed}
    \end{gathered}
\end{equation}

\subsection{Decomposing the open arm from $0$\label{sec:tower}}
 In what follows, we consider connections between good spanning clusters which lie far from both the sets $V_n$ and $D_n$. It is helpful to introduce notation for a maximal $n$-dependent scale ensuring these clusters remain distant from both $V_n$ and $D_n$. We set
 \begin{equation}
     \label{eq:maxn}
    \text{for each $n$,} \quad Q(n) := \max\left\{i: \, \left[V_n \cup D_n \right] \cap B(2^{k_{i+1}^*}) = \varnothing \right\}\ .
 \end{equation}
In the next definition, we introduce notions of localization in space for open connections betweeen good spanning clusters and for open connections between a good spanning cluster and $V_n$. 
    \begin{definition}
        \label{defin:verygood2}
        Fix a value of $i$ (which may be $0$) and a $K_2$-good spanning set $\cC \subseteq Ann_i$. We take $\cC = \{0\}$ in case $i = 0$ and often simply write $0$ instead of $\{0\}$ for simplicity. Let $\cD \subseteq Ann_{i+1}$ be another arbitrary good spanning set. We write
         \begin{equation}
            \label{eq:Fdef}
            F(\cC, \cD) := H(\cD) \cap \left\{ \cC \sa{\Z^d \setminus [S_i \cup D_n]} \cD \right\} \cap \left\{\cD \sa{\Z^d \setminus [S_{i+1} \cup D_n]} V_n  \right\}\ .
        \end{equation}
        and
        \begin{align*}
            G(\cC, \cD) = &H(\cD) \cap \left\{ \cC \sa{\Z^d \setminus S_i} \cD \right\} \cap \left\{\cD \sa{\Z^d \setminus [S_{i+1} \cup D_n]} V_n \right\}\\
            &\setminus \{\cC \sa{\Z^d \setminus \cD} \partial S_{i+1}\}\ .
        \end{align*}
        We set $F(\cC) = \cup_{\cD} F(\cC, \cD)$ and $G(\cC) := \cup_{\cD} G(\cC, \cD)$, where the union is over $K_2$-good spanning sets in $Ann_{i+1}$; we note further that for $i \leq Q(n)$, we have $F(\cC, \cD) \supseteq G(\cC, \cD)$.
    \end{definition}
    Informally speaking, the event $F(\cC)$ describes the existence of a good spanning cluster in the annulus $Ann_{i+1}$ which lies on an arm from $\cC$ to $V_n$, while $G(\cC)$ describes its existence and uniqueness.
     We note that there is an implicit dependence on $i$ in the above definitions, but since the value of $i$ is as usual uniquely determined by the choice of $\cC$, it may be safely omitted from our notation.

  We will now show that, when a good spanning set $\cC \subseteq Ann_i$ is connected via an open arm to $V_n$ off $D_n$, then portions of the cluster of this open arm are very likely to be good spanning sets, and that these good spanning sets are likely to have their further open connections localized in space. In particular, open connections realizing the event $\{0 \sa{\Z^d \setminus D_n} V_n\}$ satisfy these properties. We first, in Lemma~\ref{lem:goodlikely}, prove this fact in the sense that $\prob_{p_c}(F(\cC) \mid \cC \sa{\Z^d \setminus [D_n \cup S_i]} V_n)$ is large.

     \begin{lemma} \label{lem:goodlikely}
     There exists a $c > 0$ such that the following holds, uniformly in $p \geq p_c$, in $k_1$, in $n \geq 1$, in $0 \leq i \leq \min\{Q(n), \beta(p)\} - 1$, and in good spanning sets $\cC$  in $Ann_i$:
		\[\prob_{p}\left(F(\cC) \left| H(\cC) \cap \left\{\cC \xleftrightarrow{\Z^d \setminus [D_n \cup S_i]} V_n\right\}\right. \right) \geq 1 - \exp(-c\, k_{i+1})\ . \]
        We again recall that the case $\cC = \{0\}$ is included in the case $i = 0$, where the conditioning is on the event $\{0 \sa{\Z^d \setminus D_n} V_n\}$.
		\end{lemma}
   
    For the proof of Lemma~\ref{lem:goodlikely}, we use the following result, which allows us to decouple cylinder events from events of the form $\{\cC \sa{\Z^d \setminus [D_n \cup S_i]} V_n\}$.
    \begin{lemma}\label{lem:decoup}
    There exists a $c > 0$ such that the following holds, uniformly in $p \geq p_c$, in $k_1$, in $n \geq 1$,  $0 \leq i \leq j \leq Q(n)$, in $K_2$-good spanning sets $\cC$  in $Ann_i$, and in events $\Gamma_{j}$ depending only on bonds inside $B(2^{k_j^* + 32 d^4 + 1})$:
        \begin{equation}
            \prob_p\left(\Gamma_j \,\Big | H(\cC) \cap  \,\left\{\cC \xleftrightarrow{\Z^d \setminus [D_n \cup S_i]} V_n\right\}\right) \leq C \cdot 2^{(2d-3)k_j^*}\mathbb{P}_p(\Gamma_j \mid H(\cC)).
        \end{equation}
    \end{lemma}
	\begin{proof}
    Using a union bound and independence,
	    \begin{align}
	        \prob_p\left(H(\cC) \cap \Gamma_j \cap \left\{ \cC \xleftrightarrow{\Z^d \setminus [D_n \cup S_i]} V_n\right\}\right) &\leq \prob_p\left( H(\cC) \cap \Gamma_j \cap \bigcup_{z \in \partial B\left(2^{k_j^* +32 d^4 + 1}\right)} \left\{z \xleftrightarrow{\Z^d \setminus \left[D_n \cup B\left(2^{k_j^*+32 d^4 + 1}\right)\right]} V_n\right\}\right) \nonumber \\
         &\leq \sum_{z \in \partial B\left(2^{k_j^*+32 d^4 + 1}\right)} \prob_p(H(\cC) \cap \Gamma_j)\prob_p\left(z \xleftrightarrow{\Z^d \setminus \left[D_n \cup B\left(2^{k_j^*+32 d^4 + 1}\right)\right]} V_n\right) \nonumber\\
         &\leq C \cdot 2^{(d-1)k_j^*} \cdot \rho_j \cdot \prob_p(H(\cC) \cap \Gamma_j), \label{eq:rhobd}
	    \end{align}
     where 
     \begin{equation}
         \rho_j = \rho_j(p, n) = \max_{z \in \partial B\left(2^{k_j^*+32 d^4 + 1}\right)} \prob_p\left(z \xleftrightarrow{\Z^d \setminus \left[D_n \cup B\left(2^{k_j^*+32 d^4 + 1}\right)\right]} V_n\right).
         \label{eq:rhodef}
     \end{equation}
    But
    \begin{align}
     &\prob_p\left(H(\cC) \cap \left\{\cC \xleftrightarrow{\Z^d \setminus [D_n \cup S_i]} V_n\right\}\right)\nonumber\\
     \geq &\ \prob_p\left(H(\cC) \cap \left\{\cC \xleftrightarrow{B\left(2^{k_j^*+32 d^4 + 2}\right)\setminus S_i } z^*\right\} \cap \left\{ z^* \xleftrightarrow{\Z^d \setminus \left[D_n \cup B\left(2^{k_j^*+32 d^4 + 1}\right) \right]} V_n\right\}\right)\nonumber \\
     \geq &\ \prob_p\left(H(\cC) \cap \left\{\cC \xleftrightarrow{B\left(2^{k_j^*+32 d^4 + 2}\right)\setminus S_i } z^*\right\}\right)\prob_p\left( z^* \xleftrightarrow{\Z^d \setminus \left[D_n \cup B(2^{k_j^*+32 d^4 + 1}) \right]} V_n\right) \label{eq:usefkg} \\
     \geq &\ \prob_p(H(\cC))\prob_{p_c}\left(\left. \cC \xleftrightarrow{B\left(2^{k_j^*+32 d^4 + 2}\right)\setminus S_i } z^* \right | H(\cC)\right)\prob_{p}\left( z^* \xleftrightarrow{\Z^d \setminus \left[D_n \cup B(2^{k_j^*+32 d^4 + 1}) \right]} V_n\right) \nonumber\\
     \geq &\ c\, 2^{(2-d)k_j^*}\,\prob_p(H(\cC))  \rho_j\ .  \label{eq:usehs}
    \end{align}
    Here $z^*$ is a maximizing vertex in the definition of \eqref{eq:rhodef}; in \eqref{eq:usefkg} we used the generalized Harris-FKG inequality, and in \eqref{eq:usehs} we used item~\ref{it:knversion} of Lemma~\ref{lem:glue}. Applying the bound \eqref{eq:usehs} with \eqref{eq:rhobd} completes the proof.
 \end{proof}
  
   We now prove Lemma~\ref{lem:goodlikely}. 
	\begin{proof}[Proof of Lemma~\ref{lem:goodlikely}]
        We decompose the event $A_n(\cC):=H(\cC) \cap \{\cC \sa{\Z^d \setminus [S_i \cup D_n]} V_n\}$, recalling again the special case $i = 0$ with $\cC = \{0\}$ and $S_i = \varnothing$. Let us consider an outcome $\omega \in A_n(\cC)$ which also lies in the event
        \[U_{i+1} := \left\{\begin{array}{c}
        \text{for each spanning cluster $\widetilde \cC$ of $Ann_{i+1}$, for each $z \in \widetilde \cC$,}\\
        \text{there is a $1 \leq q \leq 32d^4$ such that $\fC_{Ann_{i+1}^q}(z)$ is a $K_2$-good spanning cluster}
        \end{array} \right\}\ . \]
        If $\gamma$ is an arbitrarily chosen self-avoiding open path in $\omega$ with $\gamma \subseteq \Z^d \setminus [S_i \cup D_n]$ realizing the connection from $\cC$ to $V_n$ guaranteed by the event $A_n(\cC)$, then letting $z$ be the first vertex of $\gamma$ after its last exit of $S_{i+1}$, we see $\fC_{Ann_{i+1}(z)}$ spans $Ann_{i+1}$ since it contains a subsegment of $\gamma$ which crosses this annulus. 

        In particular, by the choice of $\omega$, the cluster $\fC_{Ann_{i+1}(z)}$ contains an open subset $\cD \ni z$ which can be chosen, since $\omega \in U_{i+1}$, to be a good spanning cluster. Since the open path $\gamma$ realizes an open connection from $\cC$ to $\cD$, and since the portion of $\gamma$ from $z$ onward avoids $S_{i+1}$, we have $\omega \in F(\cC, \cD)$. It follows that
        \begin{equation}
            \label{eq:allgood}
            \prob_p\left(F(\cC) \left|  A_n(\cC) \right.\right) \geq \prob_p\left(U_{i+1} \left| A_n(\cC) \right.\right)\ ;
        \end{equation}
        we prove the lemma by providing a lower bound on the right-hand side of \eqref{eq:allgood}.

       To reduce clutter for the remainder of the proof, we introduce the abbreviation $\fC^q(x) := \fC_{Ann_{i+1}^q}(x)$. We decompose the complementary event $U_{i+1}^c$, since several different things guarantee its occurrence. Let us define
        \begin{align*}
            \tE_1 &= \left\{ \begin{array}{c}
            \text{there is an $x \in \partial_{in} Ann_{i+1}^0$ and some $1 \leq q \leq 32d^4$ such that }\\ \left|\partial_{in}\fC^q(x)\right| \geq 2^{7k_{i+1} /4} \text{ and } \left|\partial_{out}\fC^q(x)\right|\geq 2^{7k_{i+1}^*/4} \\
             \text{ but $\left|\partial_{in}\fC^q(x) \cap Reg^{K_2}(\cC)\right| \leq (1/2) \left|\partial_{in}\fC^q(x)\right|$ or $\left|\partial_{out}\fC^q(x) \cap Reg^{K_2}(\cC)\right| \leq (1/2) \left|\partial_{out}\fC^q(x)\right|$}
            \end{array}\right\};\\
            \tE_2 &= \left\{\begin{array}{c}
            \text{there is an $x \in \partial_{in} Ann_{i+1}^0$ such that $x \lra \partial_{in} Ann_{i+1}$ and $x \lra \partial_{out} Ann_{i+1}$,}\\
            \text{but for some $1 \leq q\leq 32d^4$, } \left| \partial_{in} \fC^q(x)\right| > 2^{9 k_{i+1}/4} \text{ or }  \left| \partial_{out} \fC^q(x)\right| > 2^{9 k_{i+1}^*/4}
            \end{array}\right\};\\
            \tE_3 &= \left\{\begin{array}{c}
            \text{there is an $x \in \partial_{in} Ann_{i+1}^0$ such that $x \lra \partial_{in} Ann_{i+1}$ and $x \lra \partial_{out} Ann_{i+1}$,}\\
            \text{but for all $1 \leq q\leq 32d^4$, either } \left| \partial_{in} \fC^q(x)\right| < 2^{7 k_{i+1}/4} \text{ or } \left| \partial_{out} \fC^q(x)\right| < 2^{7 k_{i+1}^* /4}
            \end{array}\right\}.
        \end{align*}
        It follows immediately from the definitions that if $\tE_1^c \cap \tE_2^c \cap \tE_3^c$ occurs, then for every $z$ lying in a spanning cluster of $Ann_{i+1}$ there is some $x \in \partial_{in} \fC^0(z)$ and some $q$ such that $\fC^q(x) = \fC^q(z)$ satisfies the conditions \ref{it:good1} through \ref{it:good3} of Definition~\ref{defin:Kgood}. Choosing the minimal possible $q$ gives that $\fC^q(x)$ satisfies all of Definition~\ref{defin:Kgood}. In other words, we have shown
        \[U_{i+1} \supseteq \tE_1^c \cap \tE_2^c \cap \tE_3^c \]
        and so 
        \begin{equation}
        \label{eq:totk}
            \prob_p(U_{i+1} \mid A_n(\cC)) \geq 1 - \prob_p(\tE_1 \mid A_n(\cC)) -  \prob_p(\tE_2 \setminus \tE_1 \mid A_n(\cC)) - \prob(\tE_3 \setminus \tE_2 \mid A_n(\cC)).
        \end{equation}

        We will prove the lemma by providing an appropriate bound on each negative term from~\eqref{eq:totk}.
        In fact, applying Lemma~\ref{lem:decoup} immediately gives us
        \begin{align}
            \nonumber\prob_p(U_{i+1} \mid A_n(\cC)) &\geq 1 - 2^{(2d-3) k_{i+1}^*} \left[\prob_p(\tE_1 \mid H(\cC) )-  \prob_p(\tE_2 \setminus \tE_1 \mid H(\cC)) - \prob_p(\tE_3 \setminus \tE_2 \mid H(\cC))\right]\\
            &= 1 - 2^{(4d^3-6d^2) k_{i+1}} \left[\prob_p(\tE_1) +  \prob_p(\tE_2 \setminus \tE_1 ) + \prob(\tE_3 \setminus \tE_2 )\right],\label{eq:totk2}
        \end{align}
        where in \eqref{eq:totk2} we used the independence of events depending on edges in disjoint regions. In particular, it suffices to upper bound the probabilities appearing in the square brackets of \eqref{eq:totk2}, and so we spend the remainder of the proof doing this. More specifically, we prove
        \begin{equation}
            \label{eq:totk3}
            \prob_p(\tE_1) +  \prob_p(\tE_2 \setminus \tE_1 ) + \prob_p(\tE_3 \setminus \tE_2 ) \leq C 2^{- 4d^4 k_{i+1}}
        \end{equation}
        for some constant $C$ uniform in the same parameters as in the statement of the lemma, from which the lemma's result follows.

        As before, since $i+1 \leq \beta(p)$, and since each of the events appearing in \eqref{eq:totk3} depend on the status of edges inside $B(2^{k_{i+2}^*})$, the result \eqref{eq:totk3} follows for general $p \geq p_c$ once we have proved it for the case $p = p_c$. We restrict to this case in what follows. We begin by controlling $\prob_{p_c}(\tE_1)$ by a union bound. We write
        \begin{align}
            \prob_{p_c}(\tE_1) \leq \sum_{q = 1}^{2d}\sum_{x \in \partial_{in} Ann_{i+1}^0} \Big[ &\prob_{p_c}\left(\left|\partial_{in}\fC^q(x)\right| \geq 2^{7k_{i+1}/4}, \, 
             \left|\partial_{in}\fC^q(x) \cap Reg^{K_2}(\cC)\right| \leq (1/2) \left|\partial_{in}\fC^q(x)\right|\right)\nonumber\\
             + &\prob_{p_c}\left(\left|\partial_{out}\fC^q(x)\right| \geq 2^{7k_{i+1}^*/4}, \, 
             \left|\partial_{out}\fC^q(x) \cap Reg^{K_2}(\cC)\right| \leq (1/2) \left|\partial_{out}\fC^q(x)\right|\right)\nonumber\Big],
        \end{align}
        and applying Lemma~\ref{lem:reg} gives an upper bound of
        \begin{equation}
             \prob_{p_c}(\tE_1) \leq C \exp(-c k_{i+1}^2 ) \leq C 2^{-4d^4 k_{i+1}}. \label{eq:gamma1conc}
        \end{equation}

        We continue with our estimates, turning to the bound for the second term of~\eqref{eq:totk3}. We again use a union bound:
        \begin{equation}
            \label{eq:gammaempty}
            \tE_2 \setminus \tE_1 \subseteq \bigcup_{\substack{x \in \partial_{in} Ann_{i+1}^0\\ 1 \leq q \leq 32d^4}} \left[\tE_{2, q, x}^{int} \cup \tE_{2, q, x}^{ext}\right]\ ,
        \end{equation}
        where
        \begin{align}
           \tE_{2,q, x}^{int} &:= \left\{  2^{9k_{i+1}/4} <  \left|\partial_{in}\fC^q(x)\right| \leq 2 \left|\partial_{in}\fC^q(x) \cap Reg^{(i, q,K_2)}(x)\right| \right\}\ ,\nonumber\\
           \tE_{2, q, x}^{ext} &:= \left\{  2^{9k_{i+1}^*/4} <  \left|\partial_{out}\fC^q(x)\right| \leq 2 \left|\partial_{out}\fC^q(x) \cap Reg^{(i, q,K_2)}(x)\right| \right\}\ . \nonumber
        \end{align}

        It is the case that
        \begin{equation}
            \label{eq:gammaempty3}
           \text{there is an $i_0$ such that if $i \geq i_0$, we have for each $q$ and $x$ that}\quad \tE_{2, q, x}^{int} = \varnothing = \tE_{2,q, x}^{ext}.
        \end{equation}
        The $i_0$ can actually be chosen independent of $k_1$ --- all that is required is that $2^{k_{i_0}}$ is larger than some universal constant, as we will see at \eqref{eq:Ebdd} and \eqref{eq:Ebdd2}.
        The fact \eqref{eq:gammaempty3} in conjunction with \eqref{eq:gammaempty} establishes
        \begin{equation}
            \label{eq:gammaempty4}
            \prob(\tE_2 \setminus \tE_1) \leq C 2^{- 4 d^4 k_{i+1}} \quad \text{for all $i \geq 1$}
        \end{equation}
         by choosing $C$ sufficiently large.
        
        We prove the portion of \eqref{eq:gammaempty3} involving $\tE_{2,q, x}^{int}$; the other argument is similar.  We argue that the event $\{\fC^q(x) = A\}$ for any realization $A$ with $|\partial_{in}A| \geq 2^{9k_{i+1}/4}$ and $|\partial_{in}A\cap Reg^{(i+1, q,K_2)}(A)| \geq 2^{(9 k_{i+1}/4) -1}$ has contradictory properties, and hence must be empty. 
        
        Indeed, we note  
        \[ \tE_{2,q, x}^{int} \subseteq \bigcup_{A} H(A)\ , \]
        where the union is over $A$ as in the previous sentence. We fix an $A$ as in this display and let $y \in Reg^{i+1, q, K_2}(A)$ be arbitrarily chosen. By Definition~\ref{defin:regdef}, considering contributions from the event $\mathcal{T}_{2^{k_{i+1}}}$ and its complement, we have
        \begin{align}
            \E_{p_c} \left[ \left|\fC(y) \cap B(2^{k_{i+1}})\right| \mid H(A)\right] &\leq 
            C 2^{dk_{i+1}}\left[1 - \prob_{p_c}\left(\mathcal{T}_{2^{k_{i+1}}}(y)\right) \mid H(A)\right] + C 2^{4 k_{i+1}} k_{i+1}^7\nonumber\\
            &\leq C 2^{4 k_{i+1}} k_{i+1}^7\ .\label{eq:Ebdd} 
        \end{align}

        On the other hand, letting $R = 2^{k_{i+1} - 32d^4 - 3}$ and applying item~\ref{it:toorig} of Lemma~\ref{lem:glue}, we have
        \[\text{for each $z \in B(R),$ we have } \prob_{p_c}\left(y \sa{B(2^{\ell_i})} z \mid H(A)\right) \geq cR^{17/4-d}\ . \]
        Summing over such $z$ yields
        \begin{equation}
            \label{eq:Ebdd2}
            \E_{p_c} \left[ \left.\left|\fC(y) \cap B(2^{k_{i+1}})\right|\  \right| H(A)\right] \geq cR^{17/4} \geq c 2^{(17/4) k_{i+1}}\ ,
        \end{equation}
        in contradiction to \eqref{eq:Ebdd} as long as $k_{i+1}$ is sufficiently large. This proves \eqref{eq:gammaempty3}, hence \eqref{eq:gammaempty4}.

        Finally, we upper-bound the last probability appearing in \eqref{eq:totk3}. This again involves a union bound. We write
        \begin{equation}
            \label{eq:gamma31}
            \prob_{p_c}(\tE_3 \setminus \tE_2) \leq \sum_{x \in \partial_{in} Ann_{i+1}^0} \left[\prob_{p_c}(\tE_{3, x}^{int}) + \prob_{p_c}(\tE_{3,x}^{ext})\right]\ ,
        \end{equation}
        where
        \begin{align}
            \label{eq:gamma32}
             \tE^{int}_{3, x} &:= \left\{\begin{array}{c}
            \text{$x \lra \partial_{in} Ann_{i+1}$ and $x \lra \partial_{out} Ann_{i+1}$, and for }\\
            \text{at least $16 d^4$ values of $1 \leq q\leq 32d^4$, }
            \left| \partial_{in} \fC^q(x) \right| < 2^{7 k_{i+1} /4}
            \end{array}\right\}   \setminus \tE_2 \;\\
            \tE^{ext}_{3, x} &:= \left\{\begin{array}{c}
            \text{$x \lra \partial_{in} Ann_{i+1}$ and $x \lra \partial_{out} Ann_{i+1}$, and for }\\
            \text{at least $16d^4$ values of $1 \leq q\leq 32d^4$, }
            \left| \partial_{out} \fC^q(x) \right| < 2^{7 k_{i+1}^*/4}
            \end{array}\right\}   \setminus \tE_2 \ . \nonumber
        \end{align}
        Again, we give the argument which bounds $\prob_{p_c}(\tE_{3, x}^{int})$, with the other case being similar.

        We let $\{q_k\} = \Xi \subseteq \{1, 2, \dots, 32d^4\}$ be an arbitrary subset with $|\Xi| \geq 16 d^4$; the set $\Xi$ will represent a realization of values of $q$ for which $\left|\partial_{in} \fC^q(x)\right| < 2^{7k_{i+1}/4}$. We define the event
        \[	F_q(x) = \left\{0 < \left|\partial_{in}\fC^q(x)\right| \leq 2^{7k_{i+1}/4}  \right\} \cap \left\{0 < \left|\partial_{out}\fC^q(x)\right| \leq 2^{9k_{i+1}^*/4}  \right\} ; \]
        clearly $\tE_{3,x}^{int} \subseteq \cup_{\Xi} \cap_{q \in \Xi} F_q(x)$,
        where the union is over all subsets $\Xi$ as above.
        We bound
        \begin{equation}
            \label{eq:manyqs}
            \begin{split}
            \prob_{p_c}(\tE_{3, x}^{int}) \leq \sum_{\Xi: \, |\Xi| \geq 16 d^4 } \prob_{p_c}\left(\bigcap_{q \in \Xi} F_q(x)\right)
            \end{split}
        \end{equation}
        and control the inner probability on the right uniformly in $\Xi$ and $x$ as above inductively.

        Indeed, given an arbitrary subset $\Xi \subseteq \{1, 2, \dots, 32d^4\}$ with $M := |\Xi| > 1$, we write
        \begin{equation}\
        \label{eq:toinsert} \prob_{p_c}\left(\bigcap_{q \in \Xi} F_q(x)\right) = \sum_{A} \prob_{p_c}\left(F_{q_M}(x) \mid \fC^{q_{M-1}}(x) = A\right) \prob_{p_c}( \fC^{q_{M-1}}(x) = A);
        \end{equation}
        here the sum is over cluster realizations $A$ such that $F_{q_k}(x)$ occurs for all $k < M$ (noting that these events are measurable with respect to $\fC^{q_{M-1}}(x)$). Applying Lemma~\ref{lem:nofurther}, we have
        \begin{align} 
            \prob_{p_c}\left(F_{q_M}(x) \mid \fC^{q_{M-1}}(x) = A\right)  &\leq \prob_{p_c}\left(x \lra \partial B(2^{k_{i+1} - q_M}) \mid \fC^{q_{M-1}}(x) = A\right) \nonumber 
            \\ &\leq C \left| \partial_{in}A \right| 2^{-2k_{i+1}} +  C\left| \partial_{ext}A \right| 2^{(2-d) k_{i+1}^*} 2^{dk_{i+1}}\nonumber\\
            &\leq C 2^{-k_{i+1} / 4}\ ,\label{eq:continclust}
        \end{align} 
        where we also used the one-arm probability inequality \eqref{KN}, the two-point function asymptotic \eqref{2pt}, and the cardinality bounds for $\partial_{in} A$ and $\partial_{out} A$ from the occurrence of $F_{q_{M-1}}(x)$.

        Essentially the same argument in case $|\Xi| = \ell = 1$ gives
        \[\prob_{p_c}(F_{q_1}(x)) \leq \prob_{p_c}(x \lra \partial B(2^{-k_{i+1} - 1})) \leq C 2^{-2 k_{i+1}}\ . \]
        Pulling this together with \eqref{eq:continclust} and inserting the resulting bounds into \eqref{eq:toinsert} yields
        \[ \prob_{p_c}(\tE_{3, x}^{int}) \leq C \left[2^{-k_{i+1} / 4}\right]^{16 d^4} \leq C 2^{ -4 d^4 k_{i+1}}\ .\]
        The analogous bound for $\prob(\tE_{3, x}^{ext})$ is even smaller, so we have shown that
        \begin{equation}
            \label{eq:gamma3bd}
            \prob_{p_c}(\tE_{3}) \leq C 2^{-4 d^4 k_{i+1}} 2^{(d-1) k_{i+1}} \leq C 2^{-4d^4 k_{i+1}}\ .
        \end{equation}

        Pulling together the bounds \eqref{eq:gamma1conc}, \eqref{eq:gammaempty4},\eqref{eq:gamma31}, \eqref{eq:gamma32}, and \eqref{eq:gamma3bd} shows \eqref{eq:totk3} and, as remarked just below that equation, completes the proof of the lemma.
		\end{proof}
	
	We now work to strengthen Lemma~\ref{lem:goodlikely}, showing next that $G(\cC, \cD)$ has high conditional probability given $F(\cC, \cD)$. Along the way, it is useful to note one of the most helpful properties of $G(\cC, \cD)$, namely a factorization of its probability. To express this in a reusable way, we introduce more notation.
     For each $i$ and ($K_2$-)good spanning sets $\cC \subseteq Ann_{i}$ and $\cD \subseteq Ann_{i+1}$, 
         \begin{align} \label{eq:Midef}
             M_i^p(\cC, \cD) &:= \prob_p\left( \left. H(\cD) \cap \left\{ \cC \sa{\Z^d \setminus S_i} \cD \right\}\setminus \left\{\cC\sa{\Z^d \setminus \cD} \partial S_{i+1}\right\} \right| H(\cC)\right)\ .
         \end{align}
         
         In the case $i = 0$, we introduce an analogous notation, which chiefly differs via the imposition of the cylinder event $E$.   For each good spanning set $\cC \subseteq Ann_1$, we set
         \begin{align}\label{eq:Midef2}
             M_0^p(0, \cC) &:= \prob_p\left(E \cap H(\cC) \cap \left\{ 0 \sa{S_1} \cC \right\}\setminus \left\{0 \sa{\Z^d \setminus \cC} \partial S_{1}\right\}\right)\ .
         \end{align}
         It follows immediately from independence of bond variables corresponding to disjoint sets of edges that, for $\cC \subseteq Ann_{i}$ and $\cD \subseteq Ann_{i+1}$ good spanning sets:
         \begin{equation}
             \label{eq:Gfact}
             \prob_p(G(\cC, \cD) \mid H(\cC)) =  M_i^p(\cC, \cD) \prob_p\left(\partial_{out}\cD \xleftrightarrow{\Z^d \setminus [S_{i+1} \cup D_n]} V_n \right)\ .
         \end{equation}
        
    \begin{lemma}\label{lem:noleak}
    There exist $C, c > 0$ such that the following hold uniformly in $k_1$ satisfying \eqref{eq:k0choice} and in indices $i < \min\{Q(n), \beta(p)\}$.
    For each pair $\cC \subseteq Ann_i, \, \mathcal{D} \subseteq Ann_{i+1}$ of $K_2$-good spanning sets, we have
    \begin{equation}
        \label{eq:Mslice}
        \prob_p\left( \left. H(\cD) \cap \left\{ \cC \sa{S_{i+1}} \cD \right\} \cap \left\{\cC\sa{\Z^d \setminus \cD} \partial S_{i+1}\right\} \right| H(\cC)\right) \leq C 2^{- 2d k_{i+1}}\prob_{p}(H(\cD))\ ,
    \end{equation}
    and so
    \begin{equation}
        \label{eq:Mlb}
        c |\partial_{out} \cC| |\partial_{in} \cD| 2^{(2-d) k_{i+1}} \leq \frac{M_i^p(\cC, \cD)}{\prob_{p}(H(\cD))} \leq C |\partial_{out} \cC| |\partial_{in} \cD| 2^{(2-d) k_{i+1}}.
    \end{equation}
    It follows that
        \begin{equation}\label{eq:FtoG}
            \prob_{p}\left( G(\cC, \cD) \left | H(\cC) \right. \right) \geq \prob_{p}\left( F(\cC, \cD) \left | H(\cC)  \right. \right)\left[1 - C 2^{-d k_{i+1}} \right] \ .
        \end{equation}
        
       
    \end{lemma}
    \begin{proof}
    We begin by proving \eqref{eq:Mslice}, recalling that $S_{i+1} = B(2^{\ell_{i+1}})$. We note further that the events considered in \eqref{eq:Mslice} are cylinder events, so it suffices to prove this equation at $p = p_c$ by our assumption that $i < \beta(p)$. We consider an outcome $\omega$ in the event
    \begin{equation}
    \label{eq:twist}
   \left\{ \cC \sa{S_{i+1}} \cD \right\} \cap \left\{\cC \sa{\Z^d \setminus \cD} \partial S_{i+1}\right\} \cap H(\cC)\cap H(\cD), 
    \end{equation}
    choosing a self-avoiding open path $\gamma_1$ witnessing $\left\{ \cC \sa{S_{i+1}} \cD \right\}$ and a self-avoiding open path $\gamma_2$ witnessing  $\left\{\cC \sa{\Z^d \setminus \cD} \partial S_{i+1}\right\}$. We further ensure that $\gamma_1$ and $\gamma_2$ intersect $\cC$ at only one vertex and that $\gamma_1$ intersects $\cD$ at only one vertex, noting that $\gamma_2$ does not intersect $\cD$.

      Either $\gamma_1 \subseteq B(2^{\ell_{i+1} - 1})$ or $\gamma_1$ contains a vertex outside $B(2^{\ell_{i+1} - 1})$. In the former case, we set $z \in B(2^{\ell_{i+1} - 1})$ to be the last intersection of $\gamma_1$ and $\gamma_2$ in the ordering beginning at $\cC$. Appropriate segments of $\gamma_1$ and $\gamma_2$, along with appropriate open and closed bonds of $\cC$ and $\cD$ and a connection from $\cD$ to $V_n$, witness
    \begin{equation}
        \label{eq:firstz}
   H(\cC) \circ H(\cD) \circ  \left\{ \partial_{out} \cC \xleftrightarrow{\Z^d \setminus \edges(\cC)} z \right\}\circ\left\{ z \lra \partial S_{i+1}\right\}\circ\left\{ \partial_{in} \cD \xleftrightarrow{\Z^d \setminus \edges(\cD)} z \right\}\ ,
    \end{equation}
    and so in this case, $\omega$ lies in the event described in \eqref{eq:firstz}.

      In the other case, if $\gamma_1$ exits $B(2^{\ell_{i+1} - 1})$, then the segments of $\gamma_1$ from $\cC$ to its first exit from $B(2^{\ell_{i+1} - 1})$ and from its last entry into $B(2^{\ell_{i+1} - 1})$ to $\cD$, along with appropriate open and closed bonds of $\cC$ and $\cD$, witness
    \begin{equation}
        \label{eq:secondz}
   H(\cC) \circ H(\cD)\circ \left\{ \partial_{out} \cC \xleftrightarrow{\Z^d \setminus \edges(\cC)}\partial B\left(2^{\ell_{i+1} - 1}\right) \right\}\circ\left\{ \partial_{in} \cD \xleftrightarrow{\Z^d \setminus \edges(\cD)} \partial B\left(2^{\ell_{i+1} - 1}\right) \right\}\ .
    \end{equation}
    In this case, then, our generic outcome $\omega$ from the event \eqref{eq:twist} must lie in \eqref{eq:secondz}.

    Applying \eqref{eq:firstz} and \eqref{eq:secondz} and summing, we see the probability of the event from \eqref{eq:twist} is at most
    \begin{align}
         &\prob_{p_c}(H(\cC)) \prob_{p_c}(H(\cD))\nonumber\\
         \times \Bigg[&\sum_{z \in B\left(2^{\ell_{i+1}-1}\right)} \prob_{p_c} \left(\partial_{out} \cC \xleftrightarrow{\Z^d \setminus \edges(\cC)} z \right)\prob_{p_c}\left( z \lra \partial S_{i+1} \right)\prob_{p_c}\left( \partial_{in} \cD \xleftrightarrow{\Z^d \setminus \edges(\cD)} z \right) \label{eq:err1}\\
         + &\prob_{p_c}\left( \partial_{out} \cC \xleftrightarrow{\Z^d \setminus \edges(\cC)} \partial B\left(2^{\ell_{i+1} - 1}\right) \right)\prob_{p_c}\left( \partial_{in} \cD \xleftrightarrow{\Z^d \setminus \edges(\cD)} \partial B\left(2^{\ell_{i+1} - 1}\right) \right) \Bigg]\ . \label{eq:err2}
    \end{align}
    Taking a union bound, \eqref{eq:err1} is at most
   \begin{equation*}
        \sum_{\substack{z \in B\left(2^{\ell_{i+1}-1}\right) \\ w_1 \in \partial_{out}\cC  \\ w_2 \in \partial_{in}\cD }} \prob_{p_c}(w_1 \lra z) \prob_{p_c}(z \lra w_2) \prob_{p_c}(z \lra \partial S_{i+1})\leq C 2^{-2 \ell_{i+1}} \sum_{\substack{z \in \Z^d \\ w_1 \in \partial_{out}\cC  \\ w_2 \in \partial_{in}\cD }} |z - w_1|^{2-d} |z - w_2|^{2-d}  .
        \end{equation*}
        Summing over $z$ and using \eqref{eqn: convolution}, the last display is at most
         \begin{equation}
               C 2^{-2 \ell_{i+1}} \sum_{\substack{ w_1 \in \partial_{out}\cC  \\ w_2 \in \partial_{in}\cD }} |w_1 - w_2|^{4-d} \leq C 2^{-2 \ell_{i+1}} 2^{9 k_i^*/4} 2^{25 k_{i+1}/4 - d k_{i+1}} \leq C 2^{(7-3d) k_{i+1}}\ . \label{eq:firstloop}
               \end{equation}
        Here we also used the bounds on $|\partial_{out} \cC|$ and $|\partial_{in}\cD|$ from Definition~\ref{defin:Kgood}.

        We now bound \eqref{eq:err2}, applying \eqref{KN} and the same cardinality bounds as in \eqref{eq:firstloop} to see it is at most 
        \begin{equation}
            \label{eq:secondloop}
            C\left| \partial_{out} \cC \right| \left| \partial_{in} \cD \right| 2^{-4 \ell_{i+1}} \leq C 2^{(9/4)(k_i^* + k_{i+1}) - 4 \ell_{i+1}} \leq C 2^{(18/7 - 4d) k_{i+1}}\ .
        \end{equation}
        Pulling together \eqref{eq:firstloop} and \eqref{eq:secondloop}, we bound the probability of the event in \eqref{eq:twist} by
        \begin{equation}
            \label{eq:thirdloop}
            C 2^{(18/7 - 4d) k_{i+1}}\prob_{p_c}(H(\cC)) \prob_{p_c}(H(\cD))\ .
        \end{equation}
        Dividing both sides by $\prob_{p_c}(H(\cC))$ and comparing to the bound \eqref{eq:firstloop} proves \eqref{eq:Mslice}.

 We next prove \eqref{eq:Mlb}. We use a union bound, writing
    \begin{equation} \label{eq:Munion}
    \begin{split}
        M_i^p(\cC, \cD) &\geq \prob_p\left( \left. H(\cD) \cap \left\{ \cC \sa{S_{i+1}} \cD \right\}  \right| H(\cC)\right) - \prob_p\left( \left. H(\cD) \cap \left\{ \cC \sa{S_{i+1}} \cD \right\} \cap \left\{\cC\sa{\Z^d \setminus \cD} \partial S_{i+1}\right\} \right| H(\cC)\right)\\
        &= \prob_p(H(\cD)) \prob_p\left( \left. \left\{ \cC \sa{S_{i+1}} \cD \right\}  \right| H(\cC) \cap H(\cD) \right) - \prob_p\left( \left. H(\cD) \cap \left\{ \cC \sa{S_{i+1}} \cD \right\} \cap \left\{\cC\sa{\Z^d \setminus \cD} \partial S_{ i+1}\right\} \right| H(\cC)\right)\\
        &\geq c\prob_p(H(\cD)) \left[ |\partial_{out} \cC| |\partial_{in} \cD| 2^{(2-d) k_{i+1}} -  C 2^{ - 2d k_{i+1}}\right]
        \end{split}
    \end{equation}
    where in the last line, we used Lemma~\ref{lem:glue} and \eqref{eq:Mslice}. Equation \eqref{eq:Munion} shows the first inequality of \eqref{eq:Mlb}. The upper bound of \eqref{eq:Mlb} follows directly from Lemma~\ref{lem:glue}, and so we have proved \eqref{eq:Mlb}.
    
     Finally, we show \eqref{eq:FtoG}, noting that $G(\cC, \cD) \subseteq F(\cC, \cD)$. 
     By similar reasoning to the above, the event $[F(\cC, \cD) \setminus G(\cC, \cD)]\cap H(\cC)$ is contained in a subevent of the event from \eqref{eq:twist}, namely
     \[  \left\{\cD \xleftrightarrow{\Z^d \setminus [S_{i+1} \cup D_n]} V_n  \right\} \cap   \left\{ \cC \sa{S_{i+1}} \cD \right\} \cap \{\cC \sa{\Z^d \setminus \cD} \partial S_{i+1}\} \cap H(\cC)\cap H(\cD). \]
     Therefore, taking probabilities and using independence of disjoint regions, we see 
    \begin{equation}
     \label{eqn: F-ub0}
     \begin{split}
     &\prob_{p}\left(\left[F(\cC, \cD) \setminus G(\cC, \cD)\right] \cap H(\cC)\right)\\
     \leq &\prob_{p}\left( \left\{ \cC \sa{S_{i+1}} \cD \right\} \cap \{\cC \sa{\Z^d \setminus \cD} \partial S_{i+1}\} \cap H(\cC)\cap H(\cD)\right)\prob_{p}\left(\left. \cD \xleftrightarrow{\Z^d \setminus [S_{i+1} \cup D_n]} V_n \right| H(\cD) \right).
    \end{split}
     \end{equation}
     
     Applying \eqref{eq:Mslice} in \eqref{eqn: F-ub0}, we find
     \begin{equation}
     \label{eqn: F-ub}
     \begin{split}
     &\prob_{p}\left(\left[F(\cC, \cD) \setminus G(\cC, \cD)\right] \cap H(\cC)\right)\\
    \leq~&C 2^{ - 2d k_{i+1}}\prob_{p}(H(\cD))\prob_{p}(H(\cC))\mathbb{P}_{p}\left(\left. \cD \xleftrightarrow{\Z^d \setminus [S_{i+1} \cup D_n]} V_n \right| H(\cD) \right).
    \end{split}
     \end{equation}
      On the other hand, since $i < Q(n)$, we have the following lower bound for $\prob_{p}\left( F(\cC, \cD) \cap H(\cC)   \right)$: 
        \begin{align}
              &\prob_{p}\left( H(\cC)\cap H(\cD) \cap \left\{ \cC \xleftrightarrow{\Z^d \setminus [S_i \cup D_n]} \cD \right\} \cap \left\{ \cD \xleftrightarrow{\Z^d \setminus [S_{i+1} \cup D_n]} V_n  \right\} \right)\nonumber\\
            \geq~&\prob_{p}(H(\cC)) \prob_{p}(H(\cD))\prob_{p}\left(   \left. \cC \xleftrightarrow{S_{i+1} \setminus S_i} \cD\  \right| H(\cC) \cap H(\cD)\right) \prob_{p} \left(  \cD \sa{\Z^d \setminus [S_{i+1} \cup D_n]} V_n \mid H(\cD) \right)\nonumber\\
            \geq C&2^{(7/4 + 2 - d)k_{i+1} + (7/4) k_i^*}\prob_{p}(H(\cC)) \prob_{p}(H(\cD)) \prob_{p} \left(\left. \cD \xleftrightarrow{\Z^d \setminus [S_{i+1} \cup D_n]} V_n \ \right| H(\cD) \right)  \label{eq:useglue}\\
            \geq C &2^{(15/4 - d)k_{i+1}} \prob_{p}(H(\cC)) \prob_{p}(H(\cD)) \prob_{p} \left(\left. \cD \xleftrightarrow{\Z^d \setminus [S_{i+1} \cup D_n]} V_n \ \right| H(\cD) \right) \nonumber \ .
        \end{align}
        In \eqref{eq:useglue}, we used Lemma~\ref{lem:glue}.

        Combining this lower bound with the upper bound in \eqref{eqn: F-ub} and writing
        \begin{align*}
        \mathbb{P}_{p}(G(\cC,\cD)\mid H(\cC))&=\mathbb{P}_{p}(F(\cC,\cD)\mid H(\cC))-\mathbb{P}_{p}(F(\cC,\cD) \setminus G(\cC,\cD)\mid H(\cC))\\
&= \mathbb{P}_{p}(F(\cC,\cD)\mid H(\cC))\left(1-\frac{\mathbb{P}_{p}(F(\cC,\cD) \setminus G(\cC,\cD)\mid H(\cC))}{\mathbb{P}_{p}(F(\cC,\cD)\mid H(\cC))}\right),
        \end{align*}
        we obtain \eqref{eq:FtoG}.
\end{proof}

    A final useful fact is that the event $G(\cC)$ is in fact a disjoint union of the events $G(\cC, \cD)$; this is the content of the following lemma.
    \begin{lemma}  \label{lem:Gdisj}
        Fix an $i$ and a $K_2$-good spanning set $\cC \subseteq Ann_i$. If $\cD_1$ and $\cD_2$ are two distinct $K_2$-good spanning sets in $Ann_{i+1}$, then
        \[H(\cC) \cap G(\cC, \cD_1) \cap G(\cC, \cD_2) = \varnothing\ . \]
    \end{lemma}
    \begin{proof}
        We consider a hypothetical outcome $\omega$ of $H(\cC) \cap G(\cC, \cD_1) \cap G(\cC, \cD_2)$ and show $\omega$ has contradictory properties, thereby proving the lemma.
        We first note that by item~\ref{it:good4} of Definition~\ref{defin:Kgood}, the sets $\cD_1$ and $\cD_2$ must be themselves disjoint. Indeed, $\omega \in H(\cD_1)\cap H(\cD_2)$, and so both $\cD_1$ and $\cD_2$ are connected subgraphs of the open graph in annuli, say $Ann_{i+1}^{r_1}$ and $Ann_{i+1}^{r_2}$ respectively; since $\cD_1$ and $\cD_2$ are distinct, this implies $r_1 \neq r_2$. Assuming without loss of generality that $r_1 > r_2$, connectedness again implies $\cD_1 \supseteq \cD_2$, but this contradicts item~\ref{it:good4} of Definition~\ref{defin:Kgood}.

        We go on to prove the lemma. We consider an open path $\gamma$ in $\omega$ from $\cC$ to $\cD_1$, and (by relabeling if necessary) assume that $\gamma \cap \cD_2 = \varnothing$. We adjoin to $\gamma$ a path in $\cD_1$ from the endpoint of $\gamma$ to $\partial S_{i+1}$.  By the previous paragraph and the choice of $\gamma$, this path witnesses the event $ \{\cC \sa{\Z^d \setminus \cD_2} \partial S_{i+1}\}$, but this contradicts the fact that $\omega \in G(\cC, \cD_2)$. The lemma is proved.
    \end{proof}

  \section{Initial decomposition}
  \label{sec: initial}

        In what follows, we continue the above convention and let ``good spanning set'' be an abbreviation for ``$K_2$-good spanning set''. For any $i$ and good spanning sets $\cC \subseteq Ann_i$ and $\cD \subseteq Ann_{i+1}$, recall the event $F(\cC, \cD)$ and its counterpart $F(\cC)$ defined at \eqref{eq:Fdef}.
       We recall as well the convention that $\{0\}$, written shorthand as simply $0$, is a good spanning set. 

        We recall that to prove Theorem~\ref{thm:main}, we aim to show the estimate \eqref{eq:firstsimp}; for Theorem~\ref{thm:main2}, we show a similar bound. We continue to fix a single but arbitrary cylinder event $E$; in our estimates below, some constants will depend on the particular choice of $E$ (or essentially equivalently, on the value of $L$).
        Given sequences of sets $(V_n)$ and $(D_n)$ as in the statement of Theorem~\ref{thm:main}, we decompose the probability $\prob_{p}(E, \, 0 \sa{\Z^d \setminus D_n} V_n)$ for large $n$ into a sequence of transitions between good spanning clusters. The decomposition will take the form of a large matrix product. Many of the factors in this matrix product consist of the matrices $M_i^p$ defined at \eqref{eq:Midef} and \eqref{eq:Midef2} above.

        We define the other terms now: for each $i \geq 1$, and for each good spanning set $\cC \subseteq Ann_i$, we set
        \[ \gamma_i^p(\cC, n) :=  \prob_p\left(\left. \left\{\cC \sa{\Z^d \setminus [S_{i} \cup D_n]} V_n \right\} \right| H(\cC) \right)\ .\]
         Furthermore, for each good spanning set $\cC \subseteq Ann_1$, we define
         \begin{align*}
             \widehat M_0^p(0, \cC) &:= \prob_p\left(H(\cC) \cap \left\{ 0 \sa{S_1} \cC \right\}\setminus \left\{0 \sa{\Z^d \setminus \cC} \partial S_{1}\right\}\right)\ .
         \end{align*}
         In this language, we can rephrase \eqref{eq:Gfact} as
          \begin{equation}
             \label{eq:Gfact2}
             \prob_p(G(\cC, \cD) \mid H(\cC)) =  M_i^p(\cC, \cD) \gamma_{i+1}(\cD, n)\ .
         \end{equation}
        
         Our main decomposition result is the following:
         \begin{lemma}
             \label{lem:matprod}
             There is a $c > 0$ possibly depending on $E$ but uniform in the choice of $k_1$, in $n \geq 1$, and in $p > p_c$ such that the following holds.
             Suppose $1 \leq j <  \min\{Q(n), \beta(p)\}$. Then
             \[  \prod_{i=1}^j\left[1 - \exp(-c k_{i})\right]\leq \frac{\prob_{p}\left(E, 0 \sa{\Z^d \setminus D_n} V_n\right)} {\sum_{\cC_1, \dots \cC_j}  M_0^{p}(0, \cC_1) M_1^{p}(\cC_1, \cC_2) \dots M_{j-1}^p(\cC_{j-1}, \cC_j)  \gamma_j^{p}(\cC_j, n)} \leq \prod_{i=1}^j\left[1 + \exp(-c k_{i})\right]\ \]
             and for the same $c > 0$,
             \[  \prod_{i=1}^j\left[1 - \exp(-c k_{i})\right]\leq \frac{\prob_{p}\left(0 \sa{\Z^d \setminus D_n} V_n\right)} {\sum_{\cC_1, \dots \cC_j}  \widehat M_0^{p}(0, \cC_1) M_1^{p}(\cC_1, \cC_2) \dots M_{j-1}^p(\cC_{j-1}, \cC_j)  \gamma_j^{p}(\cC_j, n)} \leq \prod_{i=1}^j\left[1 + \exp(-c k_{i})\right]\ \]
             where the sums are over good spanning sets $\cC_i \subseteq Ann_i$ for $1 \leq i \leq j$.
         \end{lemma}
        We begin with the proof of Lemma~\ref{lem:matprod}, and afterwards derive an important consequence relevant for our main results.
        \begin{proof}
        We set $\mathcal{E}$ to be either $E$ or the sure event to cover the cases of both of the displays in the lemma statement.
        We begin with the first proper annulus, $Ann_1$, assuming $1 < \min\{Q(n), \beta(p)\}$. By Lemma~\ref{lem:goodlikely} above, we have
        \begin{equation}
        \label{eq:step0}
        \begin{split}
            0 \leq \prob_{p}\left(\mathcal{E}, 0 \sa{\Z^d \setminus D_n} V_n\right) - \prob_{p}\left(\mathcal{E},  F(0)\right) &\leq \prob_{p}\left( \left\{ 0 \sa{\Z^d \setminus D_n} V_n \right\} \setminus F(0)\right)\\
            &\leq \exp(-c k_1) \prob_{p}\left(0 \sa{\Z^d \setminus D_n} V_n \right)\ .
            \end{split}
        \end{equation}
       The equation \eqref{eq:step0} is the first step toward \eqref{eq:firstsimp} and therefore toward Theorem~\ref{thm:main}.

        Our next step is to replace $F(0)$ by $G(0)$ in  \eqref{eq:step0}.  We have
        \begin{align}
       \prob_{p}(\mathcal{E}, F(0)) &\leq \sum_{\cC} \prob_{p}(\mathcal{E}, F(0, \cC))\nonumber\\
       &\leq \sum_{\cC} \left[\prob_{p}(\mathcal{E},  G(0, \cC)) + \prob_{p}( F(0, \cC) \setminus G(0, \cC))  \right]\nonumber\\
       &\leq \sum_{\cC} (1 + \exp(-c k_1)) \prob_{p}(\mathcal{E},  G(0, \cC))\ . \label{eq:oopsg}
        \end{align}
         In the last step above, we used Lemma~\ref{lem:noleak}; the constant $c$ here may depend on $E$. 
        Similarly, Lemma~\ref{lem:Gdisj} above states that if $\cC_1$ and $\cC_2$ are distinct good spanning sets of $Ann_1$, then $G(0, \cC_1) \cap G(0, \cC_2) = \varnothing$. Thus
        \begin{align}
        \prob_{p}(\mathcal{E}, F(0)) &= \prob_{p}\left(\mathcal{E}, \, \bigcup_\cC F(0, \cC)\right)\nonumber\\
        &\geq \prob_{p}\left(\mathcal{E}, \, \bigcup_\cC G(0, \cC)\right)\nonumber\\
        &= \sum_{\cC} \prob_{p}\left(\mathcal{E}, \, G(0, \cC)\right)\ .\label{eq:gtele}
        \end{align}

        Pulling these and our earlier estimates together gives that
        \begin{equation}\label{eq:step0g}
           0 \leq \prob_{p}\left(\mathcal{E}, 0 \sa{\Z^d \setminus D_n} V_n\right) - \sum_{\cC} \prob_{p}\left(\mathcal{E}, \, G(0, \cC)\right) \leq \exp(-c k_1) \prob_{p}\left(\mathcal{E}, 0 \sa{\Z^d \setminus D_n} V_n\right)\ ,
        \end{equation} 
        where the sum is over good spanning sets $\cC \subseteq Ann_1$; this is the refinement of \eqref{eq:step0}.  
        
         We would like to decouple the portions of the events in \eqref{eq:step0g} in distinct spatial regions to allow further decomposition. As at \eqref{eq:Gfact}, we may use the independence of bonds in disjoint regions to rewrite $ \prob_p\left(\mathcal{E}, G(0, \cC)\right)$ as
        \begin{align*}
          \prob_p\left(\mathcal{E} \cap H(\cC) \cap \left\{ 0 \sa{S_1} \cC \right\}\setminus \left\{0 \sa{\Z^d \setminus \cC} \partial S_{1}\right\}\right) \prob_p\left(\left. \left\{\cC \sa{\Z^d \setminus [S_{1} \cup D_n]} V_n \right\}\ \right|\  H(\cC) \right)\ .
        \end{align*}
        The first factor above is exactly $M_0^p(0, \cC)$ if $\mathcal{E} = E$ and exactly $\widehat M_0^p(0, \cC)$ if $\mathcal{E}$ is the sure event. The second factor above is $\gamma_1^p(\cC, n)$.

        Our estimate at \eqref{eq:step0g} thus guarantees the existence of a $c > 0$ uniform in $k_1$ such that
       \begin{align}
       \label{eq:casej1}
           1 - \exp(-c k_1)\leq &\frac{\prob_{p_c}\left(E, 0 \sa{\Z^d \setminus D_n} V_n\right)} {\sum_{\cC}  M_0^{p_c}(0, \cC)  \gamma_1^{p_c}(\cC, n)} \leq 1 + \exp(-c k_1) \ ,
       \end{align}
       and 
       \[ 1 - \exp(-c k_1)\leq \frac{\prob_{p_c}\left(0 \sa{\Z^d \setminus D_n} V_n\right)} {\sum_{\cC}  \widehat M_0^{p_c}(0, \cC)  \gamma_1^{p_c}(\cC, n)} \leq 1 + \exp(-c k_1) \ , \]
       This proves the lemma in the case $j = 1$.

       The remaining cases of the lemma can be proved inductively. Assuming that the lemma holds for $j$, we prove it for $j + 1$. Taking values of $n$ and $p$ such that $j + 1 < \min\{Q(n), \beta(p)\}$, we consider the quantity
       \[ \gamma_j^{p_c}(\cC_j, n) = \prob_p\left( \left\{\cC_j \left. \xleftrightarrow{\Z^d \setminus [S_{j} \cup D_n]} V_n \right\}\  \right| H(\cC) \right)\ , \]
       and further decompose the probability on the right side. An essentially identical sequence of manipulations to those which produced \eqref{eq:oopsg} and \eqref{eq:gtele} show
       \begin{align*}
           \left|\gamma_j^{p_c}(\cC_j, n) - \sum_{\cC_{j+1}} \prob_p\left( G(\cC_j, \cC_{j+1}) \mid H(\cC)\right)\right| \leq 1 + \exp(-c k_{j+1})
       \end{align*}
       where $c$ may be chosen the same as in \eqref{eq:step0g}.

       Finally, expressing $\prob(G(\cC_j, \cC_{j+1}) \mid H(\cC_j))$ using \eqref{eq:Gfact2}, the last display may be written as
        \begin{align*}
           \left|\gamma_j^{p_c}(\cC_j, n) - \sum_{\cC_{j+1}}M_{j}^p(\cC_j, \cC_{j+1}) \gamma_{j+1}^p(\cC_{j+1}, n) \right| \leq 1 + \exp(-c k_{j+1})
       \end{align*}
        By the same reasoning as at \eqref{eq:casej1}, this leads to the claim of the lemma in case $j+1$, indeed with the same choice of $c$ as at \eqref{eq:casej1}. This completes the induction and hence the proof.
        \end{proof}

    The following lemma  will allow us to complete the proof of the main theorem, conditional on a hypothesis that we will verify in the next section.
    \begin{lemma} \label{lem:givenhopf} 
    Suppose that for each choice of $k_1$ satisfying \eqref{eq:k0choice} and each pair $\cC, \cD$ of ($K_2$-)good spanning sets in $Ann_1$, there exists a real number $\alpha(\cC, \cD)$ such that
        \begin{equation}
            \label{eq:matrixfrac}
            \lim_{j \to \infty} \max_{\cC_j} \left| \frac{\sum_{\cC_2, \ldots \cC_{j-1}} M_1^{p_c}(\cC, \cC_2) \dots M_{j-1}^{p_c}(\cC_{j-1}, \cC_j)}{\sum_{\cC_2, \ldots \cC_{j-1}} M_1^{p_c}(\cD, \cC_2) \dots M_{j-1}^{p_c}(\cC_{j-1}, \cC_j)} - \alpha(\cC, \cD)\right| = 0\ ,
        \end{equation} 
        where the maximum and sums are over good spanning sets $\cC_2 \subseteq Ann_2, \dots, \cC_j \subseteq Ann_j$ respectively. Under this assumption, the following facts hold.
        \begin{enumerate}
        \item \label{it:firstgiven} The limit in \eqref{eq:thm1}  exists. As noted at \eqref{eq:thm1},  Theorem~\ref{thm:main} follows. 
        \item \label{it:secondgiven} The limit in \eqref{eq:thm2} exists. As noted at \eqref{eq:thm2},  Theorem~\ref{thm:main2} follows. 
        \end{enumerate}
    \end{lemma}
    \begin{proof}
    We begin by proving item~\ref{it:firstgiven} of the lemma by showing \eqref{eq:thm1} for arbitrary fixed sequences $(V_n, D_n)$ as in the statement of Theorem~\ref{thm:main} and an arbitrary fixed cylinder event $E$.
        We let $\delta > 0$ be arbitrary; we choose $k_1$ satisfying \eqref{eq:k0choice} large enough that, for all $n$ and $p \geq p_c$ and all $j < \min\{Q(n), \beta(p)\}$,
        \begin{align}
             e^{-\delta} \leq \frac{\prob_{p}\left(E, 0 \sa{\Z^d \setminus D_n} V_n\right)}{\sum_{\cC_1, \dots \cC_j}  M_0^{p}(0, \cC_1) M_1^{p}(\cC_1, \cC_2) \dots M_{j-1}^p(\cC_{j-1}, \cC_j)  \gamma_j^{p}(\cC_j, n)} \leq e^\delta\label{eq:Earm}
        \end{align}
        and
        \begin{align}
            e^{-\delta} \leq  \frac{\prob_{p}\left(0 \sa{\Z^d \setminus D_n} V_n\right)} {\sum_{\cC_1, \dots \cC_j}  \widehat M_0^{p}(0, \cC_1) M_1^{p}(\cC_1, \cC_2) \dots M_{j-1}^p(\cC_{j-1}, \cC_j)  \gamma_j^{p}(\cC_j, n)} \leq e^{\delta}\ .
             \label{eq:justarm}
        \end{align}
        This is possible using Lemma~\ref{lem:matprod}.

        For this choice of $k_1$, we choose $j$ large enough and then $p_0$ satisfying $\beta(p_0) > j$ so that, for all good spanning clusters $\cC, \cD \subseteq Ann_1$ and all good spanning clusters $\cC_j \subseteq Ann_j$, for all $p \in [p_c, p_0)$,
        \begin{equation}\label{eqn: ratio-l}
            e^{-\delta} \alpha(\cC, \cD) \leq \frac{\sum_{\cC_2, \ldots \cC_{j-1}} M_1^p(\cC, \cC_2) \dots M_{j-1}^p(\cC_{j-1}, \cC_j)}{\sum_{\cC_2, \ldots \cC_{j-1}} M_1^p(\cD, \cC_2) \dots M_{j-1}^p(\cC_{j-1}, \cC_j)}  \leq e^\delta \alpha(\cC, \cD)\ .
        \end{equation}
        Here we used the assumption of the lemma and the fact that each entry of $M_j^p$ is the probability of a cylinder event, hence continuous in $p$.

        Using \eqref{eq:Earm} and \eqref{eq:justarm}, we fix $p \in [p_c, p_0)$ and write
        \begin{align*}
           \prob_p\left(E \left| 0 \sa{\Z^d \setminus D_n} V_n\right.\right) &= \frac{ \prob_p\left(E \cap \left\{ 0 \sa{\Z^d \setminus D_n} V_n\right\} \right)}{\prob_{p}\left(0 \sa{\Z^d \setminus D_n} V_n\right)}\\
           &\leq e^{2 \delta} \frac{\sum_{\cC_1, \dots \cC_j}  M_0^{p}(0, \cC_1) M_1^{p}(\cC_1, \cC_2) \dots M_{j-1}^p(\cC_{j-1}, \cC_j)  \gamma_j^{p}(\cC_j, n)}{\sum_{\cC_1, \dots \cC_j}  \widehat M_0^{p}(0, \cC_1) M_1^{p}(\cC_1, \cC_2) \dots M_{j-1}^p(\cC_{j-1}, \cC_j)  \gamma_j^{p}(\cC_j, n)} \\
           &= e^{2 \delta} \sum_{\cC_1} \left[\frac{  M_0^{p}(0, \cC_1) \sum_{\cC_2, \dots \cC_j}M_1^{p}(\cC_1, \cC_2) \dots M_{j-1}^p(\cC_{j-1}, \cC_j)  \gamma_j^{p}(\cC_j, n)}{\sum_{\cC_1'}  \widehat M_0^{p}(0, \cC_1') \sum_{\cC_2', \dots \cC_j'}M_1^{p}(\cC_1', \cC_2') \dots M_{j-1}^p(\cC_{j-1}', \cC_j')  \gamma_j^{p}(\cC_j', n)}\right] \\
           &\leq e^{3 \delta} \sum_{\cC_1} \left[\frac{  M_0^{p}(0, \cC_1) \sum_{\cC_2, \dots \cC_j}M_1^{p}(\cC_1, \cC_2) \dots M_{j-1}^p(\cC_{j-1}, \cC_j)  \gamma_j^{p}(\cC_j, n)}{\sum_{\cC_1'}  \widehat M_0^{p}(0, \cC_1') \alpha(\cC_1', \cC_1) \sum_{\cC_2', \dots \cC_j'}M_1^{p}(\cC_1, \cC_2') \dots M_{j-1}^p(\cC_{j-1}', \cC_j')  \gamma_j^{p}(\cC_j', n)}\right] \\
           &= e^{3 \delta} \sum_{\cC_1} \frac{  M_0^{p}(0, \cC_1)}{\sum_{\cC_1'}  \widehat M_0^{p}(0, \cC_1') \alpha(\cC_1', \cC_1)}\ .
        \end{align*}

        In the second-to-last step, we have used \eqref{eqn: ratio-l}. The right-hand side of the last display does not depend on $n$, and so we have
        \begin{equation}
            \label{eq:upmat}
            \text{for $p \in [p_c, p_0),$}\quad\limsup_{n \to \infty} \prob_p\left(E \left| 0 \sa{\Z^d \setminus D_n} V_n\right.\right) \leq e^{3 \delta} \sum_{\cC_1} \frac{  M_0^{p}(0, \cC_1)}{\sum_{\cC_1'}  \widehat M_0^{p}(0, \cC_1') \alpha(\cC_1', \cC_1)}\ .  
        \end{equation}
        A similar string of inequalities gives
        \begin{equation}
            \label{eq:downmat}
            \text{for $p \in [p_c, p_0),$}\quad\liminf_{n \to \infty} \prob_p\left(E \left| 0 \sa{\Z^d \setminus D_n} V_n\right.\right) \geq e^{-3 \delta} \sum_{\cC_1} \frac{  M_0^{p}(0, \cC_1)}{\sum_{\cC_1'}  \widehat M_0^{p}(0, \cC_1') \alpha(\cC_1', \cC_1)}\ .  
        \end{equation}
        Combining \eqref{eq:downmat} and \eqref{eq:upmat} at $p = p_c$ and using the fact that $\delta > 0$ was arbitrary shows \eqref{eq:thm1}, proving one of the claims of the lemma.

        To show item~\ref{it:secondgiven} of the lemma by proving \eqref{eq:thm2}, we again begin by letting $\delta > 0$ be arbitrary. We return to \eqref{eq:upmat} and \eqref{eq:downmat} for this choice of $\delta > 0$, recalling that $p_0$ and $j$ were selected as functions of $\delta$ and that $j < \beta(p_0)$.  These equations were proved for arbitrary sequences $(V_n, D_n)$ as in the statement of Theorem~\ref{thm:main}; we apply them in the particular case that  $V_n = \partial B(n)$ and $D_n = \varnothing$ for each $n$.

        Again using continuity of $p \mapsto M_0^{p}(0, \cC_1)$, we can (by decreasing $p_0$ if necessary) ensure that for all $p \in [p_c, p_0)$ and all good spanning clusters $\cC_1 \subseteq Ann_1$, we have
        \begin{equation}
            \label{eq:M1cont}
            e^{-\delta} M_0^{p_c}(0, \cC_1) \leq M_0^{p}(0, \cC_1) \leq e^{\delta} M_0^{p_c}(0, \cC_1)\ . 
        \end{equation}
        Applying the choice $V_n = \partial B(n)$ and $D_n = \varnothing$  with \eqref{eq:upmat} and \eqref{eq:downmat} shows that for any $p \in (p_c, p_0)$ and for all large $n$,
        \begin{align*}
            e^{-4 \delta} \prob_p(0 \lra \partial B(n))\sum_{\cC_1} \frac{  M_0^{p}(0, \cC_1)}{\sum_{\cC_1'}  \widehat M_0^{p}(0, \cC_1') \alpha(\cC_1', \cC_1)} &\leq \prob_p\left(E \cap \{0 \lra \partial B(n)\}\right)\\
            &\leq e^{4 \delta} \prob_p(0 \lra \partial B(n))\sum_{\cC_1} \frac{  M_0^{p}(0, \cC_1)}{\sum_{\cC_1'}  \widehat M_0^{p}(0, \cC_1') \alpha(\cC_1', \cC_1)}\ , 
        \end{align*} 
        and taking $n \to \infty$ yields
        \begin{align*}
            e^{-4 \delta} \sum_{\cC_1} \frac{  M_0^{p}(0, \cC_1)}{\sum_{\cC_1'}  \widehat M_0^{p}(0, \cC_1') \alpha(\cC_1', \cC_1)} &\leq \prob_p\left(E \mid 0 \lra \infty \right)\\
            &\leq e^{4 \delta} \sum_{\cC_1} \frac{  M_0^{p}(0, \cC_1)}{\sum_{\cC_1'}  \widehat M_0^{p}(0, \cC_1') \alpha(\cC_1', \cC_1)}\ .
        \end{align*} 

        Finally, taking $p \searrow p_c$ and using \eqref{eq:M1cont} yields
        \begin{align*}
            e^{-6 \delta} \sum_{\cC_1} \frac{  M_0^{p_c}(0, \cC_1)}{\sum_{\cC_1'}  \widehat M_0^{p_c}(0, \cC_1') \alpha(\cC_1', \cC_1)} &\leq\liminf_{p \searrow p_c} \prob_p\left(E \mid 0 \lra \infty \right) \leq \limsup_{p \searrow p_c} \prob_p\left(E \mid 0 \lra \infty \right)\\
            &\leq e^{6 \delta} \sum_{\cC_1} \frac{  M_0^{p_c}(0, \cC_1)}{\sum_{\cC_1'}  \widehat M_0^{p_c}(0, \cC_1') \alpha(\cC_1', \cC_1)}\ .
        \end{align*} 
        Since $\delta > 0$ was arbitrary, this shows \eqref{eq:thm2} and completes the proof of the lemma.
    \end{proof}

 \section{Kesten's Markov Chain argument}
 \label{sec: MC-kesten}
As already noted, Lemma~\ref{lem:givenhopf} will suffice to prove Theorems~\ref{thm:main} and \ref{thm:main2} once we have verified its hypothesis. This will be the main focus of this section. We follow in the spirit of the arguments of Kesten~\cite{K}, using a lemma of Hopf \cite{H} (Lemma~\ref{lem:hopflem} below). The inputs to Hopf's lemma differ substantially from Kesten's setting, where they follow fairly directly from the Russo-Seymour-Welsh theorem valid in two-dimensional percolation. Our setup and work in the earlier parts of this paper allows us to handle the differences between the two- and high-dimensional settings.

 In what follows, we recall the definition of the relative oscillation: if $f, g$ are positive functions on a finite index space $\mathcal{S}$, we set
 \[\mathrm{osc}(f,g) := \max_{s_1, s_2 \in \mathcal{S}} \left[\frac{f(s_1)}{g(s_1)} - \frac{f(s_2)}{g(s_2)} \right]\ . \]
\begin{lemma}[\cite{H}, Theorem 1] \label{lem:hopflem}
    Let $I$, $J$ be two finite index sets and suppose ${T(i, j)}_{i \in I, j \in J}$ is a nonnegative kernel, and for any real function $f$ on $J$, we write $Tf(i) := \sum_{j} T(i, j) f(j)$. Suppose $T$ obeys the following comparability condition:
    \begin{equation}\label{eq:kappasquare}
    \frac{T(i, j) T(i', j')}{T(i, j') T(i', j)} \leq \kappa^2 \quad \text{for all $i, j, i', j'$} 
    \end{equation}
    for some $\kappa > 1$. Then we have for any positive functions $f, g$ on $J$ that
    \[\mathrm{osc}(Tf, Tg) \leq \frac{\kappa - 1}{\kappa + 1} \mathrm{osc}(f, g)\ .  \]
\end{lemma}
We will use a standard observation about oscillation: letting $T$ be a nonnegative kernel as above,
\begin{align*}
\max_{i} \frac{ Tf(i)}{Tg(i)} &= \max_i  \frac{\sum_j T(i, j)f(j)}{\sum_{j} T(i, j) g(j)} \leq \left[\max_j \frac{f(j)}{g(j)} \right]\max_i  \frac{\sum_j T(i, j) g(j)}{\sum_{j} T(i, j) g(j)} =  \max_j\left[ \frac{f(j)}{g(j)} \right]\ ,
\end{align*}
and similarly $\min_i (Tf(i) / Tg(i))$ is at least $\min_i (f(i) / g(i))$. As a consequence,
if $(T_m)_{m \geq 1}$ is a sequence of nonnegative kernels, we have
\begin{equation}
    \label{eq:osctoconv}
    \begin{gathered}
            \lim_{m \to \infty} \mathrm{osc}\left([T_m \dots T_1 f], [T_m \dots T_1 g]\right) \to 0 \\
            \text{implies there exists a $C$ such that}\\
            \lim_{m \to \infty}\max_i  \left| \frac{T_m \dots T_1 f(i)}{T_m \dots T_1 g(i)} - C  \right| = 0\ .
    \end{gathered}
\end{equation}

 \begin{proof}[Proof of Theorems~\ref{thm:main} and \ref{thm:main2}]
     As discussed above, the results will follow immediately once we have established the hypothesis \eqref{eq:matrixfrac} of Lemma~\ref{lem:givenhopf}. This hypothesis in turn follows from Lemma~\ref{lem:hopflem} using our estimates on good clusters from Lemma~\ref{lem:glue} above.

     Indeed, let $j \geq 1$. We will show there exists a constant $\kappa > 1$ playing the role of $\kappa$ from \eqref{eq:kappasquare} when $T$ is replaced by $M$. Formally, there is a $\kappa > 1$ such that for all choices of the initial scale $k_1$ satisfying \eqref{eq:k0choice},
     \begin{align*}
         \frac{M_j^{p_c}(\cC, \cD) M_j^{p_c}(\cC', \cD')}{M_j^{p_c}(\cC, \cD') M_j^{p_c}(\cC', \cD)} \leq \kappa^2 \quad \text{uniformly in $\cC, \cC', \cD, \cD'$,} 
     \end{align*}
     where $\cC, \cC'$ are generic good spanning sets in $Ann_{j}$ and $\cD, \cD'$ good spanning sets in $Ann_{j+1}$.
    Indeed, this immediately follows from the equation \eqref{eq:Mlb} of Lemma~\ref{lem:noleak}, letting $\kappa = C^2 / c^2$ for $c, C$ as in \eqref{eq:Mlb}.

    We use this and the Hopf lemma, Lemma~\ref{lem:hopflem}, to study the ratio appearing on the left of~\eqref{eq:matrixfrac}.
    We introduce the following notation for the following matrix product of $M$s, similar to the one appearing in the statement of Lemma~\ref{lem:givenhopf}:
    \begin{equation}
        \label{eq:kernelprod}
     M_{(i, j)}^p(\cC_i, \cC_j) := \sum_{\cC_{i+1}, \dots, \cC_{j-1}} M_i^p(\cC_i, \cC_{i+1}) \dots M_{j-1}^p(\cC_{j-1}, \cC_j)\ .
     \end{equation}
    In \eqref{eq:kernelprod}, the sum is over ($K_2$-)good spanning sets $\cC_{i+1} \subseteq Ann_{i+1}$, \dots, $\cC_{j-1} \subseteq Ann_{j-1}$.
    
    Lemma~\ref{lem:hopflem} shows that, uniformly in $\cC, \cC'$ good spanning sets in $Ann_1$,
    \begin{align*}
    \max_{\cC_k, \cC_k'} \left[\frac{M_{(1, k)}^{p_c}(\cC, \cC_k)}{M_{(1, k)}^{p_c}(\cC', \cC_k)} - \frac{M_{(1, k)}^{p_c}(\cC, \cC_k')}{M_{(1, k)}^{p_c}(\cC', \cC_k')} \right] \leq \left[\frac{\kappa - 1}{\kappa + 1}\right] \max_{\cC_{k-1}, \cC_{k-1}'} \left[\frac{M_{(1, k-1)}^{p_c}(\cC, \cC_{k-1})}{M_{(1, k-1)}^{p_c}(\cC', \cC_{k-1})} - \frac{M_{(1, k-1)}^{p_c}(\cC, \cC_{k-1}')}{M_{(1, k-1)}^{p_c}(\cC', \cC_{k-1}')} \right]\ ,
    \end{align*}
    where the maximum on the left-hand side is over good spanning sets in $Ann_k$, and the maximum on the right-hand side is over good spanning sets in $Ann_{k-1}$. Inducting yields
    \begin{equation}\label{eq:finalhopf}
    \begin{split}
    \max_{\cC_k, \cC_k'} \left[\frac{M_{(1, k)}^{p_c}(\cC, \cC_k)}{M_{(1, k)}^{p_c}(\cC', \cC_k)} - \frac{M_{(1, k)}^{p_c}(\cC, \cC_k')}{M_{(1, k)}^{p_c}(\cC', \cC_k')} \right] &\leq \left[\frac{\kappa - 1}{\kappa + 1}\right]^{k-2}\max_{\cC_{2}, \cC_{2}'} \left[\frac{M_1^{p_c}(\cC, \cC_{2})}{M_1^{p_c}(\cC', \cC_{2})} - \frac{M_1^{p_c}(\cC, \cC_2')}{M_1^{p_c}(\cC', \cC_2')} \right]\\
    &\leq C \left[\frac{\kappa - 1}{\kappa + 1}\right]^{k-2}
    \end{split}
    \end{equation}
    uniformly in $\cC, \cC'$ good spanning sets in $Ann_1$.
    In particular,
    \begin{equation}
        \label{eq:finalhopf2}
        \mathrm{osc}\left( M_{(1, k)}^{p_c}(\cC, \cdot),\  M_{(1, k)}^{p_c}(\cC', \cdot)  \right) \to 0 \quad \text{for each fixed pair $\cC, \cC'$.}
    \end{equation}

    Recalling the definition \eqref{eq:kernelprod} and using the observation \eqref{eq:osctoconv}, the oscillation result \eqref{eq:finalhopf2} implies the hypothesis \eqref{eq:matrixfrac} of Lemma~\ref{lem:givenhopf} holds.
   As already noted, this suffices to complete the proofs of Theorems~\ref{thm:main} and \ref{thm:main2}.
 \end{proof}

\bigskip
{\bf Acknowledgements.} The authors thank M.~Heydenreich and A.~Sapozhnikov for helpful comments and suggestions. The research of S.~C.~was supported by NSF grant DMS-2154564. The research of J.~H.~was supported by NSF grant DMS-1954257. The research of P.S. was supported by NSF grants DMS-2154090 and DMS-2238423.

 \end{document}